\documentclass[11pt]{amsart}
\usepackage[T1]{fontenc}
\usepackage{times}
\usepackage{amssymb,verbatim}
\usepackage[all]{xy}
\usepackage{subfig}
\usepackage{tikz}
\usepackage{color}
\usepackage[normalem]{ulem}
\usepackage{hyperref}
\usepackage{wrapfig}
\usetikzlibrary{arrows}
\usetikzlibrary{intersections}

\usepackage[paperwidth=21cm,paperheight=29.7cm,left=2cm,top=2cm]{geometry}

\newcommand{\R}{\mathbb{R}}
\newcommand{\C}{\mathbb{C}}

\newcommand\Z{\mathbb{Z}}
\newcommand\Prym{\textrm{Prym}}
\newcommand{\N}{\mathbb{N}}
\newcommand{\Q}{\mathbb{Q}}
\renewcommand{\H}{\mathcal H}

\newcommand{\SL}{{\rm SL}}
\newcommand{\Sur}{S}
\newcommand{\GL}{{\rm GL}}

\newcommand{\id}{\mathrm{id}}
\newcommand{\Ord}{\mathcal{O}}
\newcommand{\ety}{\varnothing}
\newcommand{\eps}{\varepsilon}

\newcommand{\Id}{\mathrm{Id}}

\newcommand{\PrD}{\Omega E_D}

\renewcommand\mod{\text{ mod }}

\newcommand{\s}{\sigma}
\newcommand{\End}{\mathrm{End}}
\newcommand{\dist}{\mathbf{d}}
\newtheorem{Theorem}{Theorem}[section]
\newtheorem{MainThm}{Theorem}

\newtheorem{Corollary}[Theorem]{Corollary}
\newtheorem{Lemma}[Theorem]{Lemma}
\newtheorem{Proposition}[Theorem]{Proposition}
\newtheorem{Remark}[Theorem]{Remark}
\newtheorem{Definition}[Theorem]{Definition}

\newtheorem{Claim}{Claim}
\newtheorem{Convention}{Convention}
\begin{document}
\title[Components of Prym eigenform loci]{Connected components of Prym eigenform loci in genus three}
\author{Erwan Lanneau and Duc-Manh Nguyen}

\date{\today}

\address{
Institut Fourier, Universit\'e de Grenoble I, BP 74, 38402 Saint-Martin-d'H\`eres, France
}
\email{erwan.lanneau@ujf-grenoble.fr}
\address{
IMB Bordeaux-Universit\'e de Bordeaux, 351, Cours de la Lib\'eration, 33405 Talence Cedex, France
}
\email{duc-manh.nguyen@math.u-bordeaux1.fr}

\maketitle
\begin{abstract}
This paper is devoted to the classification of connected components of Prym eigenform loci in the strata $\H(2,2)^{\rm odd}$ and $\H(1,1,2)$
of the Abelian differential bundle $\Omega\mathcal{M}_3$ over $\mathcal{M}_3$. These loci, discovered by McMullen~\cite{Mc7},
are $\GL^+(2,\R)$-invariant submanifolds of complex dimension $3$ of $\Omega \mathcal{M}_g$ that project to the locus of Riemann surfaces whose Jacobian
variety has a factor admitting real multiplication by some quadratic order $\Ord_D$.

It turns out that these subvarieties can be classified by the discriminant $D$ of the corresponding quadratic orders.
However there algebraic varieties are not necessarily irreducible. The main result we show is that for each discriminant $D$ the corresponding
locus has one component if $D\equiv 0,4 \mod 8$, two components if $D\equiv 1 \mod 8$, and is empty if $D\equiv 5 \mod 8$.

Surprisingly our result contrasts with the case of Prym eigenform loci in the strata $\H(1,1)$ (studied by McMullen~\cite{Mc3}) which is connected for
every discriminant $D$.
\end{abstract}

\section{Introduction}

Since the work of McMullen~\cite{Mc1} it has been known that the properties of
$\mathrm{SL}_2(\R)$-orbit closure of translation surfaces are strongly related to
the endomorphisms rings of the Jacobian of the underlying Riemann surfaces
(see also~\cite{Moller2006}). The algebro-geometric approach emphasized by
McMullen is to detect affine homeomorphisms of the flat metric on the level of the
first homology group as affine homeomorphisms induce self-adjoint endomorphisms of the Jacobian variety.\medskip

Recall that an Abelian variety $\mathbb A \in \mathcal A_g$ admits real multiplication by a totally real number field $K$ of degree $g$ over $\Q$ if there exists an inclusion $K \hookrightarrow \mathrm{End}(\mathbb A)\otimes \Q$ such that for any $k\in K$, the action of $k$ is self-adjoint with respect to the polarization of $\mathbb A$. Equivalently, $\End(\mathbb A)$ contains a copy of an order $\Ord \subset K$ acting by self-adjoint endomorphism.

\subsection{Brief facts summary in the genus $2$ case} The locus
$$
\mathcal E_2 = \{ (X,\omega) \in \Omega \mathcal{M}_2 : \mathrm{Jac}(X)
\textrm{ admits real multiplication with } \omega \textrm{ as an eigenform} \},
$$
plays an important role in the classification of $\SL(2,\R)$-orbit closures in $\Omega \mathcal{M}_2$.  Here $\mathbb A=\mathrm{Jac}(X) \in \mathcal{A}_2$, $K$ is a real quadratic field, and
the endomorphism ring is canonically isomorphic to the ring of homomorphisms of
$H_1(X,\Z)$ that preserve the Hodge decomposition. The polarization comes from the
intersection form $\left(\begin{smallmatrix} J & 0 \\ 0 & J\\ \end{smallmatrix}\right)$ on the homology.

The locus $\mathcal E_2$ is actually a (disjoint) union of subvarieties indexed by the discriminants of the orders $\Ord \subset \End(\mathrm{Jac}(X))$.
Since orders in  quadratic fields (quadratic orders) are classified by their discriminant, the {\em unique} quadratic order with discriminant $D$ is denoted by
$\Ord_D$. We then define
$$
\PrD = \{ (X,\omega) \in \mathcal E_2: \omega \textrm{ is as an eigenform for real multiplication by } \Ord_D\}.
$$
The subvarieties $\PrD$ are of interest since they are $\textrm{GL}^{+}(2,\R)$-invariant submanifolds of $\Omega \mathcal{M}_2$
(see~\cite{Mc3,Mc7}). We can further stratified $\PrD$ by defining $ \PrD(\kappa)=\PrD\cap\H(\kappa)$ for  $\kappa=(2)$ or $\kappa=(1,1)$.
This defines complex submanifolds of dimension $2$ and $3$, respectively. Hence $\PrD(2)$ projects to a union of algebraic curves
(Teichm\"uller curves) in the moduli space $\mathcal M_2$.

\subsection{Components of $\PrD(1,1)$ and $\PrD(2)$}

It is well known that the set of Abelian varieties $\mathbb A \in \mathcal{A}_2$ admitting real multiplication by $\Ord_D$ with a specified faithful representation $\mathfrak{i}: \Ord_D \rightarrow \End(\mathbb A)$ is parametrized by the Hilbert modular surface $X_D:=(\mathbb{H}\times -\mathbb{H})/\SL(\Ord_D)$. In~\cite{Mc3}, it has been shown that  each $\PrD$ can be viewed as a $\C^*$-bundle over a Zariski open subset of $X_D$, and we have
$$
\mathcal E_2 = \bigcup_{D\geq 4, D\equiv 0,1 \mod 4} \PrD
$$
In particular $\PrD$ is a connected, complex suborbifold of $\Omega\mathcal{M}_2$ of dimension $3$. The fact that there is only one (connected) eigenform locus for each $D$ follows from the fact that there is only one faithful, proper, self-adjoint representation $\mathfrak{i}: \Ord_D \rightarrow M_4(\Z)$ up to conjugation by $\mathrm{Sp}(4,\Z)$ (see \cite{Mc3} Theorem 4.4).  It follows that $\PrD(1,1)$ is a Zariski open set in $\PrD$. In particular
$\PrD(1, 1)$ is connected for any quadratic discriminant $D$. \medskip

The classification of components of $\PrD(2)$ has also been obtained by McMullen~\cite{Mc4}.

\subsection{Higher genera}

In~\cite{Mc7} it is shown that analogues of $\PrD$ exist in higher genus (up to $5$).
These subvarieties of $\Omega \mathcal{M}_g$ are called {\em Prym eigenform loci}. Surfaces in a Prym eigenform locus are pairs $(X,\omega)$ such that there exists a holomorphic  involution $ \tau : X \rightarrow X$ such that $g(X) - g(Y) = 2$, where $Y=X/\langle \tau\rangle$,  $\tau^\ast\omega=-\omega$, and the {\em Prym  variety} $\Prym(X,\tau)$  admits a real multiplication with $\omega$ an eigenform (see Section~\ref{sec:background} for precise definitions). Note that the condition $g(X) - g(Y) = 2$ is needed for our discussion. For any genus, the set of Prym eigenforms whose Prym variety admits a multiplication by $\Ord_D$ will be denoted by $\PrD$, and the intersection of $\PrD$ with a stratum $\H(\kappa)$ is denoted by $\PrD(\kappa)$.\medskip

The goal of this paper is to investigate the topology of the Prym eigenform loci in the strata $\H(2,2)$ and $\H(1,1,2)$ of $\Omega \mathcal{M}_3$. It is well known that the stratum $\H(2,2)$ also has two components $\H(2,2)^{\rm odd}$ and $\H(2,2)^{\rm hyp}$. It is not difficult to see that Prym eigenforms in $\H(2,2)^{\rm hyp}$ are double covers of surfaces in $\PrD(2)$ (see Proposition~\ref{prop:Prym:H22:hyp}). Thus $\dim\PrD(2,2)^{\rm hyp}=2$, and $\PrD(2,2)^{\rm hyp}$ is a
(finite) union of $\GL^+(2,\R)$ closed orbits. On the other hand,  we have $\dim\PrD(2,2)^{\rm odd}=3$. The  stratum $\H(1,1,2)$ is connected and  we also have $\dim\PrD(1,1,2)=3$ (see Proposition~\ref{prop:neighbor}). \medskip

Our main result reveals that the situation in genus three is quite different from the one in genus. More precisely:

\begin{MainThm}
\label{thm:main}
Let $\kappa\in \{(2,2)^{\rm odd},(1,1,2)\}$. For any discriminant $D \geq 8$, with $D\equiv 0,1 \mod  4$, the loci
$\PrD(\kappa)$ are non empty if and only if $D\equiv 0,1,4 \mod  8$, and in this case they
are pairwise disjoint. Moreover the following dichotomy holds:
\begin{enumerate}
\item If $D$ is even then $\Omega E_D(\kappa)$ is connected,
\item If $D$ is odd then $\Omega E_D(\kappa)$ has exactly two connected components.
\end{enumerate}
\end{MainThm}

\begin{Remark}
\label{rk:involution}
One of the main differences between the cases of genus two and genus three is that the polarization of the Prym variety in genus three has the form $\left(\begin{smallmatrix} J & 0 \\ 0 & 2J \end{smallmatrix} \right)$, which is the reason why $\PrD(\kappa)$ is empty if $D\equiv 5 \mod 8$. \medskip

In genus two we have $\Prym(X,\tau)=\mathrm{Jac}(X)$ and the Prym involution $\tau$ must be the hyperelliptic involution which is unique, in genus three $\Prym(X,\tau)$ is only a factor of $\mathrm{Jac}(X)$, and there may be more than one Prym involution as we will see in Section~\ref{sec:uniqueness}. Thus it is not obvious that one can use the discriminant to distinguish different Prym eigenform loci.\medskip

It is also worth noticing that while $\Omega E_9(4)$ and $\Omega E_{16}(4)$ are empty (see~\cite{LanNg11}), the loci $\Omega E_9(\kappa)$ and $\Omega E_{16}(\kappa)$ do exist for $\kappa \in \{(2,2)^{\rm odd}, (1,1,2)\}$.
\end{Remark}

\subsection{Triple tori}
\label{sec:twin}

An important tool in  our proof is the use of {\em triple tori}:
\begin{enumerate}
\item We say that $(X,\omega)\in\PrD(2,2)^{\rm odd}$ admits a {\em three tori decomposition}
if there exists a triple of homologous saddle connections $\{\s_0,\s_1,\s_2\}$ on $X$ joining the two distinct zeros of $\omega$.
\item We say that $(X,\omega)\in\PrD(1,1,2)$ admits a {\em three tori decomposition} if there exist
two pairs of homologous saddle connections $\{\s_0,\s_1\}$ and $\{\s'_0,\s'_1\}$
on $X$ joining the double zero to the simple zeros such that $\{\s'_0,\s'_1\}=\tau(\{\s_0,\s_1\})$.
\end{enumerate}

If $(X,\omega)$ admits a three tori decomposition then it can be viewed as a connected sum of three
slit tori $(X_j,\omega_j), \, j=0,1,2,$ (see Figure~\ref{fig:3twins:3tori}). We will always assume that $X_0$ is preserved
by the Prym involution $\tau$ and $X_1,X_2$ are exchanged by $\tau$.
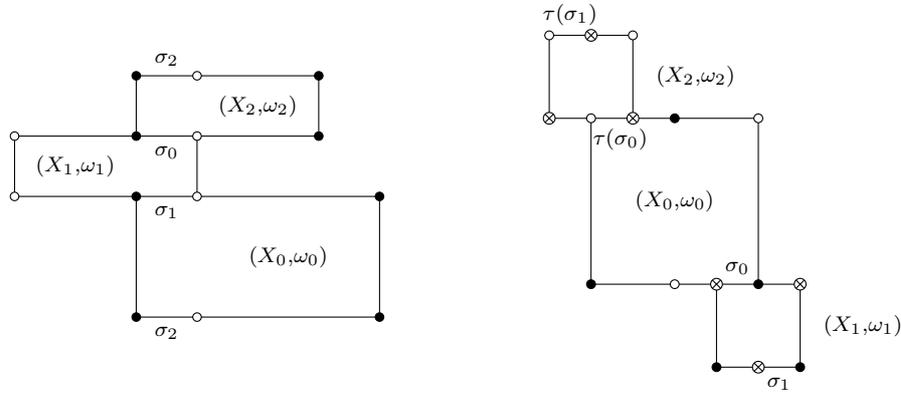
\begin{figure}[htb]
\begin{minipage}[l]{0.23\linewidth}
\centering
\begin{tikzpicture}[scale=0.8]
\draw (0,0) -- (0,2) -- (-2,2) -- (-2,3) -- (0,3) -- (0,4) (1,0) -- (4,0) -- (4,2) -- (1,2) -- (1,3) -- (3,3) -- (3,4) -- (1,4);

\foreach \x in {(0,0), (0,2), (0,3), (0,4)} \draw \x -- +(1,0);
\foreach \x in {(0,0), (0,2), (0,3), (0,4), (3,3), (3,4), (4,0), (4,2)} \filldraw[fill=black] \x circle (2pt);
\foreach \x in {(-2,2), (-2,3), (1,0), (1,2), (1,3), (1,4)} \filldraw[fill=white] \x circle (2pt);
\draw (0.5,3) node[below] {$\scriptstyle \s_0$} (0.5,2) node[below] {$\scriptstyle \s_1$} (0.5,4)
node[above] {$\scriptstyle \s_2$} (0.5,0) node[below] {$\scriptstyle \s_2$};
\draw (2.5,1) node {$\scriptstyle (X_0,\omega_0)$} (-1,2.5) node {$\scriptstyle (X_1,\omega_1)$} (2,3.5) node {$\scriptstyle (X_2,\omega_2)$};
\end{tikzpicture}
\end{minipage}
\hskip 30mm
\begin{minipage}[r]{0.23\linewidth}
\centering
\begin{tikzpicture}[scale=0.55, >=angle 60]
\draw (0,0) -- (0,4) -- (-1,4) -- (-1,6) -- (1,6) -- (1,4) -- (4,4) -- (4,0) -- (5,0) -- (5,-2) -- (3,-2) -- (3,0) -- (0,0);
\draw (3,0) -- (4,0); \draw (0,4) -- (1,4);
\foreach \x in {(0,4),(4,4),(-1,6),(1,6),(2,0)} \filldraw[fill=white] \x circle(3pt);
\foreach \x in {(2,4),(0,0),(4,0),(3,-2),(5,-2)} \filldraw[fill=black] \x circle(3pt);
\foreach \x in {(0,6),(-1,4),(1,4),(3,0),(4,-2),(5,0)}{ \filldraw[fill=white] \x circle(4pt); \draw \x +(45:4pt) -- +(225:4pt) +(135:4pt) -- +(315:4pt);}
\foreach \x in {(2,2)} \draw \x node {$\scriptstyle (X_0,\omega_0)$};
\foreach \x in {(2.5,5)} \draw \x node {$\scriptstyle (X_2,\omega_2)$};
\foreach \x in {(6.5,-1)} \draw \x node {$\scriptstyle (X_1,\omega_1)$};
\draw (-0.5,6) node[above] {$\scriptstyle \tau(\s_1)$} (4.5,-2) node[below] {$\scriptstyle \s_1$} (0.7,4)
node[below] {$\scriptstyle \tau(\s_0)$} (3.5,0) node[above] {$\scriptstyle \s_0$};
\end{tikzpicture}
\end{minipage}
\caption{Decomposition of $(X,\omega)\in \Prym(2,2)^\mathrm{odd}$ (left) and
$(X,\omega)\in \Prym(1,1,2)$ (right) into three tori.}
\label{fig:3twins:3tori}
\end{figure}

As a corollary of our main result, we prove the following theorem, which is used in the paper~\cite{Lanneau:Manh:orbits}:

\begin{MainThm}
 \label{thm:exist:3tori:dec}
 For any discriminant $D$ such that $\Omega E_D(\kappa)\neq \ety$,  there exist in any component of $\Omega E_D(\kappa)$ surfaces which admit three-tori decompositions.
\end{MainThm}

Theorem~\ref{thm:exist:3tori:dec} is proved in Section~\ref{sec:proof}.

\subsection{Strategy of the proof}

The important ingredients of the proof of the main theorem is the use of surgeries (see Section~\ref{sec:admis}).
The core of Theorem~\ref{thm:main} are Theorem~\ref{thm:no:admisSC} on admissible saddle connection
and Theorem~\ref{thm:Dodd:disconnect} on non-connectedness. The proofs of Theorem~\ref{thm:no:admisSC} and
Theorem~\ref{thm:Dodd:disconnect} appear in Section~\ref{sec:collapsing} and \ref{sec:non:connected} respectively.

\begin{enumerate}

\item An elementary way to get Prym eigenforms in $\H(2,2)^{\rm odd}$ and $\H(1,1,2)$ is given by Lemma~\ref{lm:example:D}.
Another way is the use of the surgery ``Breaking up a zero" on a Prym eigenforms in $\H(4)$ (see~\cite{Kontsevich2003}). We deduce that $\PrD(2,2)^{\rm odd}$ and $\PrD(1,1,2)$ are non-empty whenever $\PrD(4)$ is non-empty.

\item In Section~\ref{sec:uniqueness} we prove that the loci $\PrD(\kappa)$ are pairwise disjoint (Lemma~\ref{lm:order:unique} and
Theorem~\ref{th:2involutions:D9}). As we have noticed in Remark~\ref{rk:involution} a surface $(X,\omega)\in \H(2,2)^{\rm odd}$ may have
more than one Prym involution. However we then show that {\em all} Prym varieties admit real multiplication by $\Ord_9$.
This proves in particular that $\Omega E_9(2,2)^{\rm odd}$ is non-empty despite the fact that $\Omega E_9(4)=\ety$.

\item To get an upper bound of the number of components of $\PrD(\kappa)$ our strategy is to find in each component $\mathcal{C}$ of $\PrD(\kappa)$ a surface $(X,\omega)$ such that we can collapse the zeros of $\omega$ along some saddle connections to get a surface in $\PrD(4)$. Such saddle connections are called {\em admissible} (see Section~\ref{sec:admis}). In this situation, the component $\mathcal{C}$ is adjacent to the locus $\PrD(4)$
{\em i.e.} $\PrD(4)\subset \overline{\mathcal C}$. We then prove that the number of components of $\PrD(\kappa)$ that are adjacent to 
$\PrD(4)$ can not exceed the number of components of  $\PrD(4)$. Surprisingly, it turns out that there exist components that are {\em not} adjacent 
to $\PrD(4)$. These are precisely the components of the loci $\Omega E_9(\kappa)$ and $\Omega E_{16}(\kappa)$.
This fact is proved in Theorem~\ref{thm:no:admisSC}. This result plus the fact that $\PrD(4)$ has at most two connected components (by~\cite{LanNg11})
furnish the desired upper bound.

\item Finally, to get the exact count of the number of components we will show in Section~\ref{sec:non:connected} that if $\PrD(4)$ is not connected
then $\PrD(\kappa)$ is not connected either (in contrast with the situation in genus two). This difference comes from the
invariant defined in~\cite{LanNg11}.
\end{enumerate}

\subsection*{Acknowledgements}

We would like to thank Alex Eskin, Martin M\"oller, and Barak Weiss for useful discussions.
We also thank Alex Wright for useful discussions and comments on an earlier version of this text.
We would also thank Universit\'e de Bordeaux  and Institut Fourier in Grenoble for the excellent working
condition during the preparation of this work. Some of the research visits which made this collaboration possible were supported by
the ANR Project GeoDyM. The authors are partially supported by the ANR Project GeoDyM.

\section{Background}
\label{sec:background}
We review the necessary tools and results involved in the proof of our main result.
For an introduction to translation surfaces in general, and a nice survey on this topic, see {\em e.g.}~\cite{Zorich:survey, Masur2002}.

\subsection{Prym eigenform}
Let $X$ be a Riemann surface and $\tau$ an involution of $X$. We define the Prym variety of $(X,\tau)$ to be
$$
\Prym(X,\tau)=(\Omega(X,\tau)^-)^*/H_1(X,\Z)^-
$$
\noindent where $\Omega(X,\tau)^-=\ker(\tau+\id) \subset \Omega(X)$, $\Omega(X)$ is the
space of holomorphic one forms on $X$), and $H_1(X,\Z)^-$ is the anti-invariant
homology of $X$ with respect to $\tau$. Remark  that $\Prym(X,\tau)$ has naturally a polarization: the lattice $H_1(X,\Z)^{-}$  is equipped with the restriction of the intersection form on $H_1(X,\Z)$.

Following~\cite{Mc7} we will call a translation surface $(X,\omega)$ a {\em Prym form} if there exists an involution $\tau$ of $X$ such that $\dim_\C\Omega(X,\tau)^-=2$, and $\omega \in \Omega(X,\tau)^-$.
Note that the condition $\dim_\C\Omega^-(X,\tau)=2$ is equivalent to $g(X)-g(Y)=2$, where $Y:=X/\langle \tau \rangle$. In this situation, we will call $\tau$ a {\em Prym involution} of  $X$. Note that a Riemann surface may have more than one Prym involution (see Theorem~\ref{th:2involutions:D9}).

Recall that a (real) quadratic order is a ring isomorphic to $\Z[x]/(x^2+bx+c)$, the {\em discriminant} of the order is defined by $D=b^2-4c$. Orders with the same discriminants are isomorphic. Thus for any $D\in \N, \, D\equiv 0,1 \mod 4$, we will  write $\Ord_D$ to designate the {\em unique} quadratic order of discriminant $D$.  When $D$ is not a square, $\Ord_D$ is a finite index subring of the integer ring in the quadratic field $K:=\Q(\sqrt{D})$.

Let $A$ be an Abelian variety of (complex) dimension $2$. We say that $A$ admits a real multiplication by $\Ord_D$ if there exists an injective ring morphism $\mathfrak{i}: \Ord_D \rightarrow \End(A)$ such that $\mathfrak{i}(\Ord_D)$ is a self-adjoint proper subring of $\End(A)$ (properness means that if $f \in \End(A)$ and there exists $n\in \N^*$ such that $nf \in \mathfrak{i}(\Ord_D)$, then $f\in \mathfrak{i}(\Ord_D)$).

\begin{Definition}
 \label{def:eigen:form}
 We will call a translation surface $(X,\omega)$ a {\em Prym eigenform}, if there exists a Prym involution $\tau$ of $X$ such that
 \begin{itemize}
  \item[$\bullet$] $\Prym(X,\tau)$ admits a real multiplication by some quadratic order $\Ord_D$,
  \item[$\bullet$] $\omega \in \Omega(X,\tau)^-$ is an eigenvector of $\Ord_D$.
 \end{itemize}
\end{Definition}

The set of Prym eigenforms admitting real multiplication by $\Ord_D$ will be denoted by $\Omega E_D$. In \cite{Mc4} it is showed
that $\Omega E_D$ are closed, $\GL^+(2,\R)$-invariant submanifolds of the bundle $\Omega \mathcal{M}_g$. Up to now, these are the only known $\GL^+(2,\R)$-invariant submanifolds of $\Omega \mathcal{M}_g$ which are neither closed orbits nor covers of Abelian differentials or quadratic differentials in lower genus. The intersection of $\Omega E_D$ with a stratum $\H(\kappa)$ will be denoted by $\Omega E_D(\kappa)$. Clearly, $\PrD(\kappa)$ are $\GL^+(2,\R)$-invariant submanifolds of $\H(\kappa)$. \medskip

Any translation surface in genus two is a Prym form (the Prym involution being the hyperelliptic involution). It turns out that the locus
$\mathcal{E}_2$ of Prym eigenforms in genus two is a disjoint union of $\PrD$ for $D\equiv 0,1 \mod 4$ and $D\geq 5$.
It is a fact that $\PrD(2)$ is connected if $D\equiv 0,4,5 \mod 8$, and has two components otherwise ($D\equiv 1 \mod 8$).
On the other hand $\PrD(1,1)$ is connected for all $D$ (see~\cite{Mc3,Mc4}).

McMullen~\cite{Mc7} proved the existence of Prym eigenforms in genus $3$ and $4$, and in particular that  $\PrD(4)$ and $\PrD(6)$ are non-empty for infinitely many $D$. It is well known that the minimal stratum $\H(4)$ of $\Omega \mathcal{M}_3$ has two components $\H^{\rm hyp}(4)$ and $\H^{\rm odd}(4)$
(see \cite{Kontsevich2003} for precise definitions) and the loci $\PrD(4)$ are contained in $\H^{\rm odd}(4)$.
In~\cite{LanNg11} the authors  gave a complete classification of $\PrD(4)$, namely:

\begin{Theorem}[Lanneau-Nguyen~\cite{LanNg11}]
\label{thm:comp:Prym4}
For $D \geq 17$,  $\Omega E_D(4)$ is non empty if and
only  if $D\equiv 0,1,4 \mod  8$. All the loci $\Omega E_D(4)$ are pairwise disjoint.
Moreover, for the values $0,1,4$ of discriminants, the following dichotomy holds. Either
\begin{enumerate}
\item $D$ is odd and then $\Omega E_D(4)$ has exactly two connected components, or
\item $D$ is even and $\Omega E_D(4)$ is connected.
\end{enumerate}
In addition, each connected component of $\Omega E_D(4)$ corresponds to a closed $\GL^+(2,\R)$-orbit.

For $D < 17$, only $\Omega E_{12}(4)$ and $\Omega E_8(4)$ are non-empty, each of which consists of a unique closed $\GL^+(2,\R)$-orbit.
\end{Theorem}

The first striking fact about Prym eigenforms in $\H(4)$  is that $\PrD(4)=\ety$ if $D\equiv 5 \mod 8$, this is actually due to the signature of the polarization of the Prym variety in genus three which  is different from the one in genus two. The second remarkable fact is that $\Omega E_9(4)=\Omega E_{16}(4)=\ety$ even though $9$ and $16$ are not ``forbidden values'' of $D$. It is worth noticing that even though we have the same statement in the case $D \equiv 1 \mod 8$ as McMullen's result in $\PrD(2)$ (namely, $\PrD(4)$ has two components), the reason for this disconnectedness is different. Roughly speaking, the two components of $\PrD(4)$ correspond to two distinct complex lines in the space
$\Omega(X,\tau)^- \simeq H^1(X,\R)^-$, whereas in the case $\PrD(2)$, the two components correspond to the
same complex line (this is actually a consequence of the fact that $\PrD(1,1)$ is connected),  they can only be distinguished by the spin invariant (see \cite[Section 5]{Mc4}).

\subsection{Prym eigenforms in $\H(2,2)$ and $\H(1,1,2)$}

The stratum $\H(1,1,2)$ is connected while the stratum $\H(2,2)$ has two connected components:
$\H(2,2)^{\rm hyp}$ and $\H(2,2)^{\rm odd}$ (see~\cite{Kontsevich2003}). We will not use this
classification in the sequel.

\begin{Proposition}
 \label{prop:Prym:H22:hyp}
 If $(X,\omega) \in \PrD(2,2)^{\rm hyp}$ then there exists a Prym eigenform $(X',\omega') \in \Omega E_{D'}(2)$ and an unramified double cover $\rho: X \rightarrow X'$ such that $\rho^\ast\omega'=\omega$. In particular $\PrD(2,2)^{\rm hyp}$ is a finite union of $\GL^+(2,\R)$ closed orbits, each of which is a copy of a $\GL^+(2,\R)$-orbit in $\Omega E_{D'}(2)$.
\end{Proposition}

\begin{proof}[Proof of Proposition~\ref{prop:Prym:H22:hyp}]
By definition $X$ is a hyperelliptic Riemann surface, and the hyperelliptic involution $\iota$ exchanges the zeros of $\omega$. Since $\iota$ commutes with
all automorphisms of $X$, we have $\tau':=\tau\circ\iota$ is also an involution of $X$ (where $\tau$ is the Prym involution of $X$).
Set $X':=X/\langle \tau' \rangle$. Note that $\ker(\tau'-\id)=\ker(\tau+\id)$, thus we have $\dim\ker(\tau'-\id)=2$ and $X'$ is a Riemann surface of genus two.
Let $\rho: X \rightarrow X'$ be the  associated double cover. Using Riemann-Hurwitz formula, it is easy to see that $\rho$ is unramified. Since ${\tau'}^\ast\omega =\omega$, there exists a holomorphic one-form $\omega'$ on $X'$ such that $\rho^\ast\omega'=\omega$. Since $(X,\omega) \in \H(2,2)$ and
$\rho$ is unramified, we conclude that $(X',\omega') \in \H(2)$.  \medskip

Remark that $\rho^*$ is an isomorphism between $\ker(\tau'-\id)=\ker(\tau'+\id)$ and $\Omega(X')$, and $\rho$ maps $H_1(X,\Z)^-$ to a sublattice of index two in $H_1(X',\Z)$, therefore $\rho$ induces a two-to-one covering from $\Prym(X,\tau)$ to $\mathrm{Jac}(X')$. By assumption $\Prym(X,\tau)$ admits a real multiplication by the order $\Ord_D$ for which $\omega$ is an eigenvector. It follows that $\mathrm{Jac}(X')$ also admits a real multiplication by $\Ord_D\otimes\Q$ for which $\omega'$ is an eigenvector. Thus there exists a discriminant $D'$ satisfying $D'|D$ such that $(X',\omega') \in \Omega E_{D'}(2)$. This shows the
first part of the proposition.\medskip

But we know from \cite{Mc3} that $\Omega E_{D'}(2)$ is a union of $\GL^+(2,\R)$ closed orbits, and since the map $(X,\omega) \mapsto (X',\omega')$ is clearly $\GL^+(2,\R)$-equivariant, it follows that $(X,\omega)$ belongs to a $\GL^+(2,\R)$-closed orbit. Since any Riemann surface $X$ in $\mathcal{M}_g$ admits
only finitely may unramified double covers, we derive that there are only finitely many closed orbits in $\PrD(2,2)^{\rm hyp}$. The proposition is then proved.
\end{proof}

Because of Proposition~\ref{prop:Prym:H22:hyp}, in the rest of the paper we will focus on $\PrD(2,2)^{\rm odd}$
and $\PrD(1,1,2)$.  Observe that if $(X,\omega) \in \Prym(2,2)^{\rm odd}$ then the Prym involution $\tau$ exchanges the two
zeros of $\omega$, and if $(X,\omega) \in \Prym(1,1,2)$ then $\tau$ exchanges the simple zeros $\omega$.
\medskip

The next lemma provides us with examples of Prym eigenforms in $\Prym(2,2)^{\rm odd}$ and $\Prym(1,1,2)$.

\begin{Lemma}[Real multiplication by $\Ord_D$]
\label{lm:example:D}
Let $(w,h,e)\in \Z^3$ be an integral vector satisfying
$$
\left\{\begin{array}{l}        w>0,h>0,\;        \gcd   (w,h,e)    =1,\\
   D=e^2+8wh, \; e+\sqrt{D}>0.
\end{array}
\right.
$$
Let $\lambda:=\frac{e+\sqrt{D}}{2}$. Note that $\lambda^2=e\lambda+2wh$. We denote by $\Sur_{\kappa}(w,h,e)$
the surface defined in Figure~\ref{fig:examples:D} below. Then
$$
\Sur_{\kappa}(w,h,e) \in \Omega E_D(\kappa).
$$
\end{Lemma}

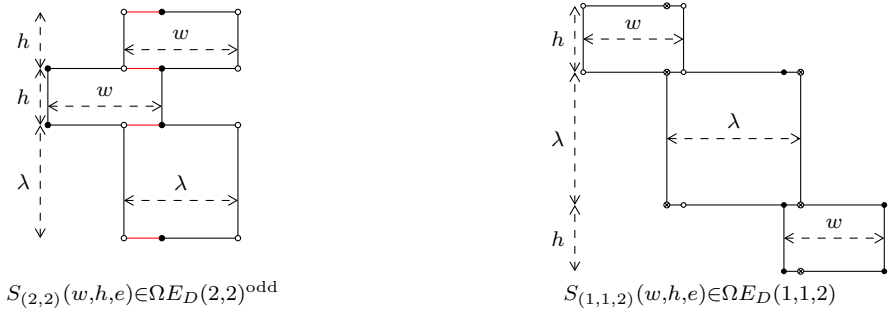
\begin{figure}[htb]
\begin{minipage}[t]{0.3\linewidth}
\centering
\begin{tikzpicture}[scale=0.5, >=angle 60]
\draw (3,0) -- (5,0) -- (5,3) -- (3,3) -- (3,4.5) -- (5,4.5) -- (5,6) -- (3,6) (2,6) -- (2,4.5) -- (0,4.5) -- (0,3) -- (2,3) -- (2,0);
\foreach \x in {(2,6), (2,4.5), (2,3), (2,0)} \draw[red] \x -- +(1,0);
\foreach \x in {(2,6), (5,6), (2,4.5), (5,4.5), (2,3), (5,3), (2,0), (5,0)} \filldraw[fill=white] \x circle (2pt);
\foreach \x in {(3,6), (0,4.5), (3,4.5), (0,3), (3,3), (3,0)} \filldraw[fill=black] \x circle (2pt);
\foreach \x in {(2,1), (0,3.5), (2,5)} \draw[<->, dashed] \x -- +(3,0);
\foreach \x in {(-0.2,3), (-0.2,4.5)} \draw[<->, dashed] \x -- +(0,1.5); \draw[<->, dashed] (-0.2,0) -- (-0.2,3);
\foreach \x in {(3.5,1)} \draw \x node[above] {$\scriptstyle \lambda$};
\foreach \x in {(1.5,3.5), (3.5,5)} \draw \x node[above] {$\scriptstyle w$};
\draw (-0.2,1.5) node[left] {$\scriptstyle \lambda$} (-0.2,3.75) node[left] {$\scriptstyle h$} (-0.2,5.25) node[left] {$\scriptstyle h$};

\draw (2.5,-1.5) node {$\scriptstyle \Sur_{(2,2)}(w,h,e) \in \Omega E_{D}(2,2)^{\rm odd}$};
\end{tikzpicture}
\end{minipage}
\begin{minipage}[t]{0.6\linewidth}
\centering
\begin{tikzpicture}[scale=0.22,>=angle 60]
\draw (0,0) -- (0,8) -- (-5,8) -- (-5,12) -- (1,12) -- (1,8) -- (8,8) -- (8,0) -- (13,0) -- (13,-4) -- (7,-4) -- (7,0) -- (0,0);
\draw (7,0) -- (8,0); \draw (0,8) -- (1,8);
\draw[<->, dashed] (0,4) -- (8,4);

\foreach \x in {(-5,8),(1,8),(-5,12),(1,12), (1,0)} \filldraw[fill=white] \x circle(4pt);
\foreach \x in {(7,8),(7,0),(13,0),(7,-4),(13,-4)} \filldraw[fill=black] \x circle(4pt);

\foreach \x in {(0,12),(0,8),(8,8),(0,0),(8,0),(8,-4)} {\filldraw[fill=white] \x circle (5pt); \draw \x +(45:5pt) -- +(225:5pt) +(135:5pt) -- +(315:5pt);}

\foreach \x in {(-5,10), (7,-2)}\draw[<->, dashed] \x -- +(6,0);

\foreach \x in {(-5.5,8), (-5.5,-4)} \draw[<->, dashed] \x -- +(0,4);
\foreach \x in {(-5.5,0)} \draw[<->, dashed] \x -- +(0,8);
\foreach \x in {(-5.5,4)} \draw \x node[left] {$\scriptstyle \lambda$};
\foreach \x in {(-5.5,10), (-5.5,-2)} \draw \x node[left] {$\scriptstyle h$};
\foreach \x in {(-2,10), (10,-2)} \draw \x node[above] {$\scriptstyle w$};
\foreach \x in {(4,4)} \draw \x node[above] {$\scriptstyle \lambda$};
\draw (2,-5.5) node {$\scriptstyle \Sur_{(1,1,2)}(w,h,e)\in \Omega E_D(1,1,2)$};
\end{tikzpicture}
\end{minipage}
\caption{Real multiplication by $\Ord_{D}$}
\label{fig:examples:D}
\end{figure}

\begin{proof}[Proof of Lemma~\ref{lm:example:D}]
Each surface  in Figure~\ref{fig:examples:D} is a connected some of three slit tori, and admits an involution $\tau$ which fixes one torus and
exchanges the other two (see also Section~\ref{sec:twin}). It is not difficult to see that $\tau$ is a Prym involution, and that $\Sur_\kappa(w,h,e) \in \Prym(\kappa)$.
Let $(X,\omega)$ be one of the surfaces in Figure~\ref{fig:examples:D}. Let $X_0$ be the torus invariant by $\tau$, and $X_1,X_2$ be the other
two tori. By construction, there are bases $(a_i,b_i)$ of $H_1(X_i,\Z)$, $i=0,1,2$, such that
\begin{itemize}
\item[$\bullet$] $\tau(a_0)=-a_0, \tau(b_0)=-b_0$, $\tau(a_1)=-a_2, \tau(b_1)=-b_2$,
 \item[$\bullet$] $\omega(a_0)=\lambda,\omega(b_0)=\imath  \lambda$,
 \item[$\bullet$] For $i=1,2$ one has $\omega(a_i)=w$ and $\omega(b_i)=\imath h$.
\end{itemize}
Set $a=a_1+a_2, b=b_1+b_2$. Then $\{a_0,b_0,a,b\}$ is a symplectic basis of $H_1(X,\Z)^-$ in which the intersection form is given by the matrix
$\left(\begin{smallmatrix} J & 0 \\ 0 & 2J \end{smallmatrix}\right)$.
Let $T$ be the endomorphism of $H_1(X,\Z)$ which is given in the basis $(a_0,b_0,a,b)$ by the matrix
$$
T=\left(\begin{array}{cc}
e\Id_2 & \left(\begin{smallmatrix} 2w & 0 \\ 0 & 2h \\ \end{smallmatrix}\right)\\
\left(\begin{smallmatrix} h & 0 \\ 0 & w \\ \end{smallmatrix}\right) & 0 \\
\end{array}\right).
$$
Since the restriction of the intersection form on $H^-_1(X,\Z)$ is given by
$\left(\begin{smallmatrix} J & 0 \\ 0 & 2J\\ \end{smallmatrix}\right)$, it is easy to check that $T$ is self-adjoint
with respect this form. Note that in this basis $\omega$ is given by the vector
$(\lambda,\imath \lambda,2w,2h\imath)$, therefore we have $T^\ast\omega= \lambda \omega$. It follows that $T \in \mathrm{End}(\Prym(X,\tau))$.
Since $T$ satisfies  $T^2 = eT + 2wh\Id_4$, $T$ generates a self-adjoint proper subring of $\mathrm{End}(\Prym(X,\tau))$ isomorphic to $\Ord_D$
for which $\omega$ is an eigenvector. Thus $\Sur_\kappa(w,h,w)\in \Omega E_D(\kappa)$. The lemma is proved.
\end{proof}

\begin{Corollary}
 \label{cor:non:empty}
For any $D \geq 8$, $D \equiv 0,1,4 \mod 8$, the loci $\Omega E_D(2,2)^{\rm odd}$ and $\Omega E_D(1,1,2)$ are non-empty.
 \end{Corollary}
\begin{proof}[Proof of Corollary~\ref{cor:non:empty}]
Apply Lemma~\ref{lm:example:D} for some $(w,h,e)\in \Z$ with $D=e^2+8wh$.
\end{proof}

\subsection{Kernel foliation}

To investigate the topology of these loci we first recall the notion of the kernel foliation. Let $(X,\omega)\in\H(\kappa)$ be a translation surface.
In a neighborhood of $(X,\omega)$ the kernel foliation leaf of $(X,\omega)$ consists of surfaces having the same absolute periods as $(X,\omega)$  and the relatives periods slightly different. This foliation has already appeared in several papers (see for example \cite{EskMasZor03, Calta2004, Masur:Zorich, survey:Farb:06, MinWei14}.

For $\kappa \in \{ (2,2)^{\rm odd}, (1,1,2)\}$, the intersection with the kernel foliation leaves  gives rise to a foliation of the Prym eigenform loci $\PrD(\kappa)$, the leaves of this foliation  have complex dimension one. Constructions of surfaces in  the intersection of the kernel foliation and Prym eigenform loci can be found in \cite{Lanneau:Manh:cp, Lanneau:Manh:orbits}. Since the leaves of this foliation has dimension one, for any $(X,\omega)\in \PrD(\kappa)$, we can use the notation $(X,\omega)+w$, with $w \in \C$ and   $|w|$ small, to denote  surfaces in the same leaf and close to $(X,\omega)$ (see~\cite{Lanneau:Manh:orbits}, Section 3). Moreover up to the action of $\GL^+(2,\R)$, a neighborhood of $(X,\omega)$ in $\PrD(\kappa)$ consists of surfaces in the same kernel foliation leaf as $(X,\omega)$. Namely, we have

\begin{Proposition}[\cite{Lanneau:Manh:cp}, Corollary 3.2]
\label{prop:neighbor}
Let $(X',\omega') \in \PrD(\kappa)$ close enough to $(X,\omega)\in \PrD(\kappa)$. Then there exists a unique pair $(g,w)$, where
$g \in \GL^+(2,\R)$ close to $\id$, and $w\in \C$ with $|w|$ small, such that $(X',\omega')=g\cdot((X,\omega)+w)$. In particular, we have
$\dim\PrD(\kappa)=3$.
\end{Proposition}

\section{Uniqueness}
\label{sec:uniqueness}

In genus two Prym involution and hyperelliptic involution coincide, so it is unique. In
higher genus, surfaces may have more than one Prym involution (see {\em e.g.} the Appendix).
In~\cite{LanNg11}, we showed that if $(X,\omega)\in \PrD(4)$ then the Prym involution is unique. This is no longer true in $\H(2,2)^{\rm odd}$, nevertheless, if a surface in $\H(2,2)^{\rm odd}$ has two Prym involutions, then both Prym varieties admit real multiplication by $\Ord_9$ as we will see in Theorem~\ref{th:2involutions:D9}. It follows in particular that if $D_1\neq D_2$ and  $\kappa\in \{(2,2)^{\rm odd},(1,1,2)\}$
then $\Omega E_{D_1}(\kappa)\cap \Omega E_{D_2}(\kappa) =\varnothing$.

\begin{Theorem}
\label{th:2involutions:D9}
Let $(X,\omega)\in \H(2,2)^{\rm odd}$ be a surface having two Prym involutions $\tau_1\neq \tau_2$ such that $\tau_1^*\omega=\tau_2^*\omega=-\omega$.
Then there exist  $\mathfrak{i}_1 : \Ord_9 \rightarrow \mathrm{End}(\mathrm{Prym}(M,\tau_1))$
and $\mathfrak{i}_2 : \Ord_9 \rightarrow \mathrm{End}(\mathrm{Prym}(M,\tau_2))$ such that $\mathfrak{i}_i(\Ord_9)$ is a self-adjoint proper subring of $\mathrm{End}(\Prym(X,\tau_i))$, and  $\omega$ is an eigenform for both subrings $\mathfrak{i}_1(\Ord_9)$ and $\mathfrak{i}_2(\Ord_9)$. In particular $(X,\omega)\in \Omega E_9(2,2)^{\rm odd}$.

If $(X,\omega) \in \H(1,1,2)$, then there exists at most one Prym involution $\tau$ such that $\tau^*\omega=-\omega$.
\end{Theorem}

\begin{Corollary}
\label{cor:2involutions}
For $\kappa \in \{(2,2)^{\rm odd},(1,1,2)\}$, if $D_1\neq D_2$ then $\Omega E_{D_1}(\kappa)\cap \Omega E_{D_2}(\kappa)=\ety$.
\end{Corollary}

\begin{proof}[Proof of Corollary~\ref{cor:2involutions}]
Let $(X,\omega) \in \Omega E_{D_1}(\kappa) \cap \Omega E_{D_2}(\kappa)$. Let  $\tau_1$ and $\tau_2$ be the corresponding Prym involutions of $X$. If $\tau_1 \neq \tau_2$ then by Theorem~\ref{th:2involutions:D9} one has $\kappa=(2,2)^{\rm odd}$ and $D_1 = D_2 = 9$. If $\tau_1 = \tau_2$ then Lemma~\ref{lm:order:unique},
applied to ${\mathbb A} = \Prym(X,\tau_1) = \Prym(X,\tau_2)$, gives the uniqueness of the self-adjoint proper subring $\Ord \subset \mathrm{End}(\mathbb{A})$, which implies $D_1=D_2$. The corollary is then proved.
\end{proof}

We will need the following two elementary lemmas. The first one proves the uniqueness of the
proper subring, once the Prym variety and the eigenform are given (see also~\cite[Section~$5$]{LanNg11}
for related results).

\begin{Lemma}
\label{lm:order:unique}
Let ${\mathbb A}$ be an Abelian variety of dimension two.
We regard ${\mathbb A}$ as a quotient $\C^2/L$, where $L$ is a lattice isomorphic to $\Z^4$ equipped with a
non-degenerate skew-symmetric inner product $\langle, \rangle: L \times L \rightarrow \Z$ which is compatible with
the complex structure. Let $v\neq 0$ be a vector in $\C^2$. Assume that there exists a self-adjoint endomorphism
$\varphi$ of ${\mathbb A}$ such that $\varphi(v)=\lambda v$, with $\lambda\in \R$, $\varphi \neq \lambda\cdot\Id$.
Then there exists a unique discriminant $D$ and a unique self-adjoint proper subring $\Ord$ of ${\rm End}({\mathbb A})$
isomorphic to $\Ord_D$ for which $v$ is an eigenvector.
\end{Lemma}

\begin{proof}
Let $S=\C\cdot v$  be the complex line generated by $v$, and let $S'$ denote the orthogonal complement of $S$ with
respect to $\langle, \rangle$ in $\C^2$. Note that $S'$ is also a complex line in $\C^2$. Set $w=\imath v$. Since $\varphi$ is an endomorphism of
${\mathbb A}$, we have $\varphi(w)=i\varphi(v)=\lambda w$. In other words $\varphi_{|S} = \lambda\cdot \id_S$.
Since $\varphi$ is self-adjoint, it also preserves the complex line $S'$. Thus $\varphi_{|S'}=\lambda'\cdot \id_{S'}$
where $\lambda'\neq \lambda$. \medskip

Since the self-adjoint endomorphism $\varphi$ preserves the lattice $L$ its minimal polynomial
$\chi_\varphi \in \Z[X]$ has degree $2$. By definition $\lambda$ is a real root of $\chi_\varphi$. Hence
$\lambda'$, that is a root, is also real. Moreover, since $v$ is an eigenvector of $\varphi$, up to a real scalar,
all the coordinates of $v$ in a basis of $L$ belong to $K=\Q(\lambda)$.
Remark that either $K=\Q$, or $K\subset \R$ and $[K:\Q]=2$. \medskip

Let $K_v$ be the subring of ${\rm End}({\mathbb A})\otimes \Q$ consisting of self-adjoint endomorphisms of ${\mathbb A}$ for which $v$ is an eigenvector. For any $f\in K_v$, the matrix of $f$ in the decomposition $\C^2=S\oplus S'$ has the form $\left(\begin{smallmatrix} \lambda(f) & 0 \\ 0 & \lambda'(f) \end{smallmatrix}\right)$.\medskip

We claim that $K_v$ is either isomorphic to $K$ or to $\Q^2$. To see this, it suffices to notice that each  element of $K_v$ is uniquely determined by its eigenvalues on $S$ and $S'$. If $\lambda\in \Q$ then we can assume that all the coordinates of $v$ belong to $\Q$, hence both $\lambda(f)$ and $\lambda'(f)$ belong to $\Q$ as $f$ is defined of over $\Q$. Thus $\Lambda: K_v \rightarrow \Q^2, \quad f\mapsto (\lambda(f),\lambda'(f))$ is an isomorphism of $\Q$ vector spaces.  If $\lambda \not\in \Q$, then $\lambda \in K=\Q(\sqrt{d})$, with $d\in \N$, $d$ is not a square.  It follows that $\lambda'(f)$ is the Galois conjugate of $\lambda(f)$ in $K$. Consequently, $\Lambda: K_v \rightarrow K, \quad f \mapsto \lambda(f)$ is an isomorphism of $\Q$-vector spaces. \medskip

Set $\Ord=K_v\cap {\rm End}({\mathbb A})$. By definition, $\Ord$ is the unique self-adjoint proper subring of $\mathrm{End}(\mathbb{A})$ for which $v$ is an eigenvector. It remains to show that $\Ord \simeq \Z[X]/(X^2+bX+c)$ for some $b,c \in \Z$, such that $b^2-4c>0$. We have $\dim_\Q K_v=2$. For any $f\in \Ord$ such that $(f,\id)$ is a basis of $K_v$ as a $\Q$-vector space, we denote by $\Delta(f)$ the discriminant of the minimal polynomial of $f$. Note that we have $\Delta(f)>0$, and $\Delta(f)=(\lambda(f)-\lambda'(f))^2$. Set $D=\min\{\Delta(f): \; f \in \Ord, \ \ (f,\id) \text{ is a basis of } K_v\}$. Let  $\psi$  be an element of $\Ord$ such that $\Delta(\psi)=D$. Let us show that $\Ord =\Z\psi+\Z\id$. Indeed, let $f$ be an element of $\Ord$, then we can write $f=x\psi+y\id$, with $(x,y) \in \Q^2$, and $\Delta(f)=x^2D$. If $x\not\in \N$, by replacing $\psi$ by $f-[x]\psi$, we can find $\psi'\in \Ord$ such that $0<\Delta(\psi')<D$, therefore we must
have $x\in \N$. It follows that $y\in\N$. Finally, since $\psi\in {\rm End}({\mathbb A})$, the minimal polynomial of $\psi$ has the the form $\psi^2+b\psi+c\id$, with $b,c\in \Z$ such that $D=b^2-4c$. The proof of the lemma is now complete.
\end{proof}

\begin{Lemma}
\label{lm:2inv:tor:cover}
Let $(X,\omega) \in \H(2,2)\sqcup\H(1,1,2)$, and $\tau_1,\tau_2$ be two Prym involutions of $X$ such that $\tau_1^*\omega=\tau_2^*\omega=-\omega$.
\begin{itemize}
 \item [(a)] If $(X,\omega) \in \H(1,1,2)$ then $\tau_1=\tau_2$.

 \item[(b)] If $(X,\omega)\in \H(2,2)^{\rm odd}$ and $\tau_1\neq \tau_2$, then there exists a branched cover
$p:X \rightarrow Y$ of degree three, where  $Y$ is a torus, which is ramified only at the zeros of $\omega$ and satisfies $\omega=p^*\xi$, where  $\xi$ is a holomorphic one-form on $Y$. Moreover, the involutions $\tau_1,\tau_2$ descend to the unique involution of $Y$ which acts by $-\id$ on the homology
and exchanges the images by $p$ of the zeros of $\omega$.
 \end{itemize}
\end{Lemma}

\begin{proof}[Proof of Lemma~\ref{lm:2inv:tor:cover}]
Set $\tau=\tau_1\circ\tau_2$, one has $\tau^*\omega=\omega$ and $\tau$ fixes all the zeros of $\omega$. We identify the neighborhood of a double zero $P$ of $\omega$ with the unit disk $\Delta \subset \C$ such that $P$ is mapped to $0$. In this local chart, $\omega=z^2dz$. If $\omega \in \H(1,1,2)$ then  $P$ is the unique double zero of $\omega$, therefore $\tau_1(P)=\tau_2(P)=P$. Since both $\tau_1,\tau_2$ are involutions, their restrictions to this neighborhood of $P$ read $\tau_i(z)=-z$. Thus $\tau(z)=z$, which implies that $\tau=\id_X$ and $\tau_1=\tau_2$. \medskip

From now on, we will assume that $\omega \in \H(2,2)^{\rm odd}$ and $\tau_1\neq \tau_2$.
The restriction of $\tau$ to the local chart of $P$ (defined above) can be written as $\tau(z)=\zeta z$. Since $\tau^*\omega=\omega$, it follows that $\zeta^3=1$. Obviously $\zeta \neq 1$, otherwise $\tau$ is the identity map in a neighborhood of $P$ and hence it is
the identity on $X$ implying $\tau_1 = \tau_2$. Let $p: X\rightarrow Y=X/\langle \tau \rangle$ be the quotient map. Since $\tau$ has order three,
$p$ is a ramified covering of degree $3$. Clearly, the two zeros of $\omega$ are branched points of $p$
of order $3$. Moreover since $\dim \ker(\tau-\id) \geq 1$, one has $ \mathrm{genus}(Y) \geq 1$. These two facts, combined
with the Riemann-Hurwitz formula
$$
-4 = 2-2\cdot \mathrm{genus}(X)=3\cdot(2-2\cdot \mathrm{genus}(Y))-\sum_{x\in X} (e_p(x)-1) \leq -\sum_{x\in X} (e_p(x)-1) \leq -4
$$
implies that $\mathrm{genus}(Y)=1$ and the two zeros of $\omega$  are the only branched points of $p$.
Since $\tau^*\omega=\omega$, the form $\omega$ descends to a holomorphic $1$-form $\xi$ on $Y$
{\em i.e.} $\omega=p^*\xi$. \medskip

Now the subgroup of $\mathrm{Aut}(X)$ generated by $\tau$, namely $\{\id, \tau_1\circ\tau_2,\tau_2\circ\tau_1\}$,
is clearly invariant by the conjugations by $\tau_1$ and $\tau_2$. Therefore $\tau_1$ and $\tau_2$ induces two involutions,
say $\iota_1$ and $\iota_2$, on $Y$. Since $\tau_i$ permutes the zeros of $\omega$, the
equality $\tau_i^*\omega=-\omega$ reads $\iota_i^*\xi=-\xi$, and $\iota_i$ exchanges the images of the zeros of $\omega$ by $p$.
Now $Y$ is a torus: there  exists only one such involution. Hence $\iota_1=\iota_2$. The lemma is then proved.
\end{proof}

\begin{proof}[Proof of Theorem~\ref{th:2involutions:D9}]
From Lemma~\ref{lm:2inv:tor:cover} it is sufficient to assume that $(X,\omega)\in \H(2,2)^{\rm odd}$. For simplicity we
continue with the notions of the previous lemma. Let $P,Q\in X$ denote the zeros of $\omega$.
We claim that there exists a basis of $H_1(Y,\Z)$ given by a pair of simple closed geodesics $\{\alpha, \beta\}$
that are invariant by $\iota$ and that do not contain $p(P)$ and $p(Q)$. Indeed one can pick a fixed point of $\iota$ and take a pair of simple
closed geodesics passing through this point and missing the point $p(P)$ and $p(Q)=\iota(p(P))$. \medskip

The next goal is to construct a symplectic basis of $X$ and a self-adjoint endomorphism of $\Prym(X,\tau_1)$.
We lift the curves $\alpha$, $\beta$ to $X$ in the following manner. Let $R\in Y$ be
the (unique) intersection point of $\alpha$ with $\beta$. Hence $\iota(R)=R$ and $R$ is a
regular point of the covering $p$.
Since $p\circ \tau_1 = \iota\circ p$, the involution $\tau_1$ induces a permutation (of order two) of
$p^{-1}(R)=\{R_0,R_1,R_2\}$. Therefore $\tau_1$ fixes some point, say $R_0$. \medskip

Let $\alpha_0$ (respectively, $\beta_0$) be the pre-image of $\alpha$ (respectively, $\beta$)
passing through $R_0$. For $j=1,2$ let $\alpha_j=\tau^j(\alpha_0), \beta_j=\tau^j(\beta_0)$, where $\tau=\tau_1\circ\tau_2$.
By construction:
$$
\left\{ \begin{array}{lll}
\langle \alpha_i, \alpha_j \rangle =\langle \beta_i, \beta_j \rangle = 0 &\textrm{ for } &i,j \in \{0, 1,2\}, i\neq j,\\
\langle \alpha_i, \beta_j \rangle =\delta_{ij} &\textrm{ for } &i,j \in \{0, 1,2\}.
\end{array} \right.
$$
Hence $(\alpha_0,\beta_0,\alpha_1,\beta_1, \alpha_2,\beta_2)$ is a symplectic basis of $H_1(X,\Z)$.
This allows us to construct a symplectic basis $(a_0,b_0,a_1,b_1)$ of $H_1(X,\Z)^{-} = \ker(\tau_1+\mathrm{id})$ as usual:
$$
\left\{\begin{array}{ll}
a_0 = \alpha_0 & b_0 = \beta_0 \\
a_1 = \alpha_1+\alpha_2 & b_1 = \beta_1+\beta_2
\end{array} \right.
$$
The intersection form is given by the matrix
$\left(\begin{smallmatrix} J & 0 \\ 0 & 2J\\ \end{smallmatrix}\right)$. One can normalize by using $\GL^+(2,\R)$ so that $\xi(\alpha)=1$ and $\xi(\beta)=\imath\in\C$.
Then $\omega$ (viewed as an element of  $H^1(X,\C)^{-}$) is the vector (in the basis dual to $(a_0,b_0,a_1,b_1)$)
$$
(\omega(a_0),\omega(b_0),\omega(a_1),\omega(b_1))   = (1,\imath,2,2\imath).
$$
It is straightforward to check that $(X,\omega)$ coincide with 
$\Sur_{(2,2)}(1,1,-1) \in \Omega E_9(2,2)$.

Let us consider the matrix
$$
T=\left(\begin{smallmatrix} -\id_2 & 2\cdot \id_2 \\ \id_2 & 0 \end{smallmatrix}\right)_{(a_0,b_0,a_1,b_1)}
$$
It is straightforward to check that $T$ is self-adjoint with respect to the restriction of the intersection form on
$H_1(M,\Z)^-$. Moreover $(1,\imath,2,2\imath)\cdot T=(1,\imath,2,2\imath)$ thus $\omega$ is an
eigenform for $T$: hence $(X,\omega) \in \PrD(2,2)^{\rm odd}$. Since $T^2=-T+2$, the endomorphism $T$ generates
a proper subring of $\mathrm{End}(\mathrm{Prym}(M,\tau_1))$ isomorphic to $\Ord_D$ where
$D=1+4\cdot 2=9$. By Lemma~\ref{lm:order:unique}, this subring is unique. \medskip

The same argument shows that $\mathrm{End}(\mathrm{Prym}(M,\tau_2))$ also contains a unique self-adjoint proper subring isomorphic
to $\Ord_9$, for which $\omega$ is an eigenform. The proof of the theorem is now complete.
\end{proof}
\section{Non connectedness of $\PrD(\kappa)$}
\label{sec:non:connected}

In this section we will show that when $D\equiv 1 \mod 8$, the number of components of $\PrD(2,2)^{\rm odd}$ and $\PrD(1,1,2)$ is at least two. It is worth noticing that this is not true in genus two, namely,  $\PrD(2)$ has two connected components, while $\PrD(1,1)$ is connected (see~\cite{Mc4}).

\begin{Theorem}
\label{thm:Dodd:disconnect}
For any $D\geq 9$ satisfying $D\equiv 1 \mod 8$, the loci $\PrD(2,2)^{\rm odd}$ and $\PrD(1,1,2)$  are not connected.
\end{Theorem}
\begin{proof}[Proof of Theorem~\ref{thm:Dodd:disconnect}]
First of all by Corollary~\ref{cor:non:empty}, $\PrD(\kappa)$ is non-empty.
Before going into the details, we first explain why $\PrD(4)$ is not connected
when $D\equiv 1 \mod 8$ (see~\cite[Theorem~6.1]{LanNg11}).
For every surface $(X,\omega)\in\PrD(4)$ we denote by $S$ the subspace of
$H^1(X,\R)^{-} = \ker(\tau+\mathrm{Id}) \subset H^1(X,\R)$ generated by $\{\mathrm{Re}(\omega),\mathrm{Im}(\omega)\}$
and by $S'$ the orthogonal complement of $S$ with respect to the intersection form in $H^1(X,\R)^{-}$.
By definition there is an injective ring morphism $\mathfrak{i} : \Ord_D \rightarrow \mathrm{End}(\Prym(X,\tau))$ such that $\mathfrak{i}(\Ord_D)$ is a self-adjoint proper subring of $\mathrm{End}(\Prym(X,\tau))$, and for any $T \in \mathfrak{i}(\Ord_D)$, $S$ is an eigenspace of $T$.
It turns out that an element $T\in \mathrm{Im}(\mathfrak{i})$ is uniquely
determined by its minimal polynomial and by the eigenvalue of its restriction to $S$.
Indeed, the minimal polynomial of $T$ has degree two; thus if $T_{|S}=\lambda\id_S$ then $T_{|S'}=\lambda'\id_{S'}$,
where $\lambda$ and $\lambda'$ are the roots of the minimal polynomial of $T$. Therefore $T$ is given by the
matrix $\left(\begin{smallmatrix} \lambda & 0 \\ 0 & \lambda' \\ \end{smallmatrix}\right)$ in the decomposition
$H^1(X,\R)^-=S\oplus S'$. In~\cite[Section 6]{LanNg11} for each $D\equiv 1 \mod 8$, with $D\geq 17$,
we constructed two surfaces $(X_i,\omega_i) \in \PrD(4)$, $i=0,1$, where the corresponding generators of the order
$T_0 \in \mathrm{Im}(\mathfrak{i}_0)$ and $T_1 \in \mathrm{Im}(\mathfrak{i}_1)$ satisfies:
\begin{itemize}
\item $T_0$ and $T_1$ have the same minimal polynomial,
\item $T_{0|S_0}=\lambda\id_{S_0}$ and $T_{1|S_1}=\lambda\id_{S_1}$,
\item $\langle,\rangle_{|\mathrm{Rg}(T_0)} \neq 0 \mod 2$ and  $\langle,\rangle_{|\mathrm{Rg}(T_1)} = 0 \mod 2$.
\end{itemize}
Now if $(X_0,\omega_0)$ and $(X_1,\omega_1)$ belong to the
same connected component ({\em i.e.} $A\cdot (X_0,\omega_0) = (X_1,\omega_1)$ where
$A\in\mathrm{GL}^+(2,\R)$) then there exists an isomorphism
$f: H_1(X_0,\Z)^- \rightarrow H_1(X_1,\Z)^-$ such that $T'_1:= f^{-1}\circ T_1 \circ f$
defines an endomorphism of $\mathrm{Prym}(X_0,\tau_0)$. By uniqueness
of the Prym involution and the map $\mathfrak{i}_0 : \Ord_D \rightarrow \mathrm{End}(X_0,\tau_0)$,
it follows that both $T_0$ and $T'_1$ belong to $\mathfrak{i}_0(\Ord_D)$.
By construction $T_0$ and $T'_1$ have the same minimal polynomial and
$T_{0|S_0}=(T'_1)_{|S_0}= \lambda\id_{S_0}$. Thus in view of the above remark, $T_0=T'_1$.
This is a contradiction since $\langle,\rangle_{|\mathrm{Rg}(T_0)} \neq \langle,\rangle_{|\mathrm{Rg}(T'_1)} \mod 2$. \medskip

We now go back to the proof of Theorem~\ref{thm:Dodd:disconnect} and apply  similar ideas.
Let $(w,h,e)\in \Z^3$ be as in Lemma~\ref{lm:example:D} where $D=e^2+8wh \equiv 1 \mod 8$. Note that $e$ is odd.
We will show that the two surfaces
$$
(X_0,\omega_0):=\Sur_\kappa(w,h,-e) \in\Omega E_D(\kappa) \qquad \textrm{and} \qquad (X_1,\omega_1):=\Sur_\kappa(w,h,e) \in\Omega E_D(\kappa)
$$
do not belong to the same component.
Recall that by construction, for $j=0,1$, we have associated to $(X_j,\omega_j)$ a generator $T_j$ of
the order $\mathfrak{i}_j(\Ord_D) \subset {\rm End}(\Prym(X_j,\tau_j))$. Recall that, in the symplectic basis
$(a_0,b_0,a,b)$ of $H_1(X_j,\Z)$ given in Lemma~\ref{lm:example:D}, the endomorphism is given by the matrix
$$
T_0=\left(\begin{array}{cc}
-e\Id_2 & \left(\begin{smallmatrix} 2w & 0 \\ 0 & 2h \\ \end{smallmatrix}\right)\\
\left(\begin{smallmatrix} h & 0 \\ 0 & w \\ \end{smallmatrix}\right) & 0 \\
\end{array}\right) \qquad \textrm{and} \qquad
T_1=\left(\begin{array}{cc}
e\Id_2 & \left(\begin{smallmatrix} 2w & 0 \\ 0 & 2h \\ \end{smallmatrix}\right)\\
\left(\begin{smallmatrix} h & 0 \\ 0 & w \\ \end{smallmatrix}\right) & 0 \\
\end{array}\right) \textrm{ respectively }.
$$
Let us assume that there is a continuous path $\gamma : [0,1] \rightarrow \Omega E_D(\kappa)$ such that $\gamma(i)=(X_i,\omega_i)$ for $i=0,1$,
we will draw a contradiction.
Let $\tilde{\gamma}$ be a lift of $\gamma$ to the vector bundle $\Omega\mathcal{T}_3$ over the Teichm\"uller
space $\mathcal{T}_3$. We will denote by $(X_s,\omega_s)$ the image of $s\in [0,1]$ by $\tilde{\gamma}$.
Let $\Sigma$ be the base surface of the Teichm\"uller space.  By construction the path $\tilde{\gamma}$ induces
a continuous map which sends every $s\in [0,1]$ to a tuple $(\mathbf{J}_s,\tau_s,L_s,\mathfrak{i}_s,S_s)$, where
\begin{itemize}
\item $\mathbf{J}_s$ is the complex structure of $H^1(\Sigma,\R)$, induced by the complex structure of $X_s$,
\item $\tau_s\in \mathrm{Sp}(6,\Z)$ is the matrix which gives the action of the Prym involution of $X_s$ on $H^1(\Sigma,\R)$,
\item $L_s$ is the lattice $H^1(\Sigma,\Z)\cap \ker(\tau_s+\id)$,
\item $\mathfrak{i}_s : \Ord_D \rightarrow \mathrm{End}(H^1(\Sigma,\R)_s^-)$, where $H^1(\Sigma,\R)_s^-=\ker(\tau_s+\id) \subset H^1(\Sigma,\R)$, is an injective ring morphism where $\mathfrak{i}_s(\Ord_D)$ is a self-adjoint proper subring  of $\mathrm{End}(H^1(\Sigma,\R)_s^-)$ that preserves $L_s$.
\item $S_s$ is the subspace of $H^1(X,\R)^{-} = \ker(\tau_s+\mathrm{Id})$ generated by $\{\mathrm{Re}(\omega_s),\mathrm{Rg}(\omega_s)\}$.
\end{itemize}
Remark that since $\omega$ is holomorphic, $S_s$ is invariant by $\mathbf{J}_s$.
The action of $\GL^+(2,\R)$ preserves the subspace $S_s\subset H^1(\Sigma,\R)^-$, and the kernel foliation leaves
invariant $[\mathrm{Re}(\omega)]$ and $[\mathrm{Im}(\omega)]$. Therefore $S_s$ is invariant along the path
$\tilde{\gamma}$. Clearly, the matrix $\tau_s$ is also invariant along the deformation $\tilde{\gamma}$.
This implies that $L_s$ and $\mathfrak{i}_s$  are  also invariant  along $\tilde{\gamma}$.
In particular $S_0=S_1=S$ and $\mathfrak{i}_0=\mathfrak{i}_1$.

There exists an isomorphism $f: H_1(X_0,\Z)^- \longrightarrow H_1(X_1,\Z)^-$
satisfying $f(S_0)=S_1$ and such that $T'_1= f^{-1}\circ T_1 \circ f$ belongs to $\mathfrak{i}_0(\Ord_D) \subset\mathrm{End}(\Prym(X_0,\tau_0))$.
Remark that $T'_1$ and $T_0 + e\cdot \mathrm{Id}$ have the same minimal polynomial $X^2-eX-2wh$.
In addition the eigenvalues of $T'_1$ on $S_0=\C\cdot\omega_0$, and
the eigenvalue of $T_0+e\cdot \mathrm{Id}$ on $\C\cdot \omega_0$ are both equal to $\lambda = (e+\sqrt{D})/2$.
Hence $T'_1= T_0+e\cdot \mathrm{Id}$. \medskip

Now $\mathrm{Rg}(T_0+e\cdot \mathrm{Id}) \mod 2$ is generated by $\{a,b\}$. The restriction of the intersection form $\langle , \rangle$ to
this subspace is equal to $0 \mod 2$. On the other hand the restriction of the intersection form to $\mathrm{Rg}(T'_1)$ does not vanish modulo $2$:
$$
\langle T_1(a_0),T_1(b_0)\rangle \equiv \langle a_0, b_0 \rangle \equiv 1 \mod 2.
$$
This is a contradiction, and the theorem follows.
\end{proof}

In Section~\ref{subsec:connectedness} we will give a topological argument for the non connectedness of $\Omega E_9(2,2)^{\rm odd}$.

\section{Collapsing zeros along a saddle connection}
\label{sec:admis}

In this section we describe a surgery on collapsing several zeros of Prym eigenforms together such that the resulting surface
is still a non degenerate Prym eigenform. This can be thought as the converse of the surgery
``breaking up a zero'' (see~\cite{Kontsevich2003}).

In what follows, all the zeros are {\em labelled} and all the saddle connections are {\em oriented}:
if $(X,\omega)\in \PrD(2,2)^{\rm odd}$, we label the zeros by $P$ and $Q$ and if
$(X,\omega)\in \PrD(1,1,2)$, we label the simple zeros by $R_1,R_2$ and the double zero by $Q$.
Let $\s_0$ be a saddle connection on $X$.
\begin{Convention}
\label{convention:saddle}
We will always assume that:
\begin{enumerate}
\item If $(X,\omega)\in \PrD(2,2)^{\rm odd}$ then $\s_0$ is a saddle connection from $P$ to $Q$ that is invariant by $\tau$.
\item If $(X,\omega)\in \PrD(1,1,2)$ then $\s_0$ is a saddle connection from $R_1$ to $Q$.
\end{enumerate}
\end{Convention}
Observe that such saddle connections always exist on any $(X,\omega)\in \PrD(\kappa)$: for $\kappa=(2,2)^{\rm odd}$, take $\s_0$ to be the union of a path of
minimal length from a regular fixed point of $\tau$ to a zero of $\omega$ and its image by $\tau$, for $\kappa=(1,1,2)$, take
a path of minimal length from the set $\{R_1,R_2\}$ to $\{Q\}$.

\subsection{Admissible saddle connections}

We begin with the following definition.

\begin{Definition}
\label{def:simple:SC}
Let $(X,\omega)\in \Prym(\kappa)$ be a Prym form.
\begin{enumerate}
\item $\kappa=(2,2)^{\rm odd}$: we say that $\s_0$ is admissible if for any saddle connection $\s\neq \s_0$ from $P$ to
$Q$ satisfying $\omega(\s)=\lambda \omega(\s_0)$, with $\lambda \in \R_+$, one has $\lambda >1$.
\item $\kappa=(1,1,2)$: we say that $\s_0$ is admissible if, for any saddle connection
$\s\neq \s_0$ starting from $R_1$ and satisfying $\omega(\s)=\lambda \omega(\s_0)$, with $\lambda \in \R_+$,
either $\lambda >1$ if $\sigma$ ends at $Q$, or $\lambda >2$ if $\s$ ends at $R_2$.
\end{enumerate}
\end{Definition}

Observe that by definition the subset consisting of surfaces having an admissible saddle connection
is an open $\GL^+(2,\R)$-invariant subset.

\begin{Lemma}
\label{lm:twin:saddles}
Let $\kappa\in \{(2,2)^{\rm odd},(1,1,2)\}$. For any $(X,\omega)\in \PrD(\kappa)$ and any
$\s_0$ satisfying Convention~\ref{convention:saddle}, there exists in a neighborhood of $(X,\omega)$ a surface
$(X',\omega')\in \PrD(\kappa)$ with a saddle connection $\s'_0$ (corresponding to $\s_0$ and also satisfies Convention~\ref{convention:saddle}) such that
any saddle connection $\s'$ on $X'$ in the same direction as $\s'_0$, if it exists, satisfies
\begin{enumerate}
\item Case $\kappa=(2,2)^{\rm odd}$: $\omega'(\s')=\omega'(\s'_0)$.
\item Case $\kappa=(1,1,2)$:
\begin{enumerate}
\item If $\s'$ connects a simple zero to the double zero then $\omega'(\s')=\omega'(\s'_0)$.
\item If $\s'$ connects two simple zeros then $\omega'(\s')=2\omega'(\s'_0)$.
\end{enumerate}
\end{enumerate}
\end{Lemma}

\begin{proof}[Proof of Lemma~\ref{lm:twin:saddles}]
For any vector $v\in\R^2$ small enough we denote by $\s'_0$ the saddle connection on
$(X',\omega')=(X,\omega)+v$ corresponding to $\s_0$. Observe that the set
$$
\mathrm{Slope}(X,\omega)=\{s \in \R\cup \{\infty\} : s \text{ is the slope of } \omega(\gamma) \neq 0, \text{ with } [\gamma]\in H_1(X,\Z)\}
$$
is countable. Hence there exists a vector $v\in\R^2$ small enough so that the slope of $\omega'(\s'_0)$
does not belong to the set $\mathrm{Slope}(X,\omega)=\mathrm{Slope}(X',\omega')$. \medskip

\noindent {\bf Case $\kappa=(2,2)^{\rm odd}$}. Let $\s'$ be a saddle connection starting from $P$ in the same direction as $\s'_0$.
If $\s'$ ends at $P$ then $[\s']\in H_1(X',\Z)$ and $0\neq \omega'(\s')$ has the same slope as $\omega'(\s'_0)$.
This is a contradiction. Thus $\s'$ ends at $Q$ and $[\gamma']=[\s'_0*(-\s')] \in H_1(X',\omega')$.
If $\omega'(\gamma')\neq 0$ then we get again a contradiction.
Therefore $\omega'(\gamma')= 0$ {\em i.e.} $\omega'(\s')=\omega'(\s'_0)$. \medskip

\noindent {\bf Case $\kappa=(1,1,2)$}. Let $\s'$ be a saddle connection starting from $R_1$ in the same direction as
$\s'_0$. If $\s'$ ends at $R_1$ we run into the same contradiction. If $\s'$ ends at $Q$ then we also run into
the same conclusion namely $\omega'(\s')=\omega'(\s'_0)$. Hence let us assume that $\s'$ ends
at $R_2$. Thus $\s'_0*\tau(\s'_0)*(-\s')$ is a closed path from $R_1$ to $R_1$ (through $Q$ and $R_2$).
Therefore $[\gamma']=[\s'_0*\tau(\s'_0)*(-\s')] \in H_1(X',\omega')$. The same contradiction shows that
$\omega'(\gamma')  = 0$, {\em i.e.} $\omega'(\s')=2\omega'(\s'_0)$. The lemma is proved.
\end{proof}

\subsection{Non-admissible saddle connection and twin/double-twin}

Lemma~\ref{lm:twin:saddles} leads to the following natural definition.

\begin{Definition}
Let $(X,\omega)\in \PrD(\kappa)$ and let $\s_0$ be a saddle connection on $X$ satisfying Convention~\ref{convention:saddle}. \\
If $\kappa=(2,2)^{\rm odd}$: a saddle connection $\s$ is a {\em twin} of $\s_0$ if $\s$ joins $P$ to $Q$ and $\omega(\s)=\omega(\s_0)$.\\
If $\kappa=(1,1,2)$:
\begin{enumerate}
\item a saddle connection $\s$ is a twin of $\s_0$ if it has the same endpoints and $\omega(\s)=\omega(\s_0)$.
\item a saddle connection $\s$ is a double-twin of $\s_0$ if $\s$ joins $R_1$ to $R_2$ and $\omega(\s)=2\omega(\s_0)$.
\end{enumerate}
\end{Definition}
When $(X,\omega)\in \PrD(2,2)^{\rm odd}$, since the angle between two twin saddle connections is a multiple of $2\pi$ and
 the angle at $P$ is $6\pi$, we see that each saddle connection $\s_0$ has at most two twins.
When $(X,\omega)\in \PrD(1,1,2)$, the same remark shows that $\s_0$ has at most one twin or one double-twin.
Moreover the midpoint of any double twin saddle connection is fixed by the Prym involution. \medskip

Non admissible saddle connection does not necessarily imply the
existence of a twin or double twin (see Remark~\ref{rk:no:twins} below). However Lemma~\ref{lm:twin:saddles}
shows that, under mild assumption, this dichotomy holds. As an immediate corollary, we draw

\begin{Corollary}
\label{cor:twin}
Let $\kappa\in \{(2,2)^{\rm odd},(1,1,2)\}$. For any connected component $\mathcal C$ of $\PrD(\kappa)$,
there exist $(X,\omega)\in \mathcal C$ and a saddle connection $\s_0$ on $(X,\omega)$
satisfying Convention~\ref{convention:saddle} such that either $\s_0$ is admissible, or
$\s_0$ has a twin or a double twin.
\end{Corollary}

\subsection{Collapsing admissible saddle connections}

We have the following proposition that is a converse to the surgery ``breaking up a zero''
(see~\cite{Kontsevich2003}). We prove the proposition only in the setting of Corollary~\ref{cor:twin}.
A more general statement holds but it is not needed in this paper.

\begin{Proposition}
\label{prop:collapse}
Let $(X,\omega)\in \PrD(\kappa)$ be a Prym form and $\s_0$ a saddle connection satisfying
Convention~\ref{convention:saddle}. We assume that $\s_0$ is admissible. Then one can collapse the zeros of $\omega$
along $\s_0$ by using the kernel foliation so that the resulting surface belongs to $\PrD(4)$.

In particular, if there is no saddle connection in the same direction as $\s_0$ connecting two different zeros, then one can collapse $\s_0$ to get a surface in $\PrD(4)$.
\end{Proposition}

\begin{proof}
We first consider the case when $(X,\omega)\in \PrD(2,2)^{\rm odd}$.
The proof we describe is constructive.
Set $\ell=|\s_0|$. As usual we assume that $\s_0$ is horizontal. By definition of admissible saddle connection, any other horizontal saddle connection from $P$ to $Q$ has length $> \ell$.

For any horizontal geodesic ray emanating from a zero of $\omega$ we say that the ray is {\em positive} if it has direction
$(1,0)$ and {\em negative} if it has direction $(-1,0)$. For instance by convention $\s_0$ is a positive
ray for $P$ and a negative ray for $Q$. Since the conical angle at $P$ and $Q$ is
$6\pi$, there are two other positive horizontal rays emanating from $P$, say $\s_{P,1}^+$ and $\s_{P,2}^+$,
as well as two other negative rays for $Q$, say $\s_{Q,1}^-$ and $\s_{Q,2}^-$.
We parametrize each ray by its length to the zero where it emanates. Obviously, if a positive ray intersects a negative ray then it corresponds to a (horizontal) saddle connection. \medskip

We will first prove the proposition under a slightly stronger condition
\begin{center}
($\mathcal{C}$) any horizontal saddle connection other than $\s_0$ has length $> 2\ell$.  
\end{center}
We will construct a set $\mathcal T$ which is a union of horizontal rays as follows:
the first element of $\mathcal T$ is $\s_0$. Now if a ray $\s_{P,1}^+$ or $\s_{P,2}^+$ intersects
$\s_{Q,1}^-$ or $\s_{Q,2}^-$ by assumption, the associated saddle connection has
length $\lambda> 2\ell$. Hence we can choose $\eps>0$ such that positive rays
$\s_{P,1}^+$ and $\s_{P,2}^+$ at time $\ell+\eps$ do not intersect $\s_{Q,1}^-$ and
$\s_{Q,2}^-$ at time $\ell+\eps$ in their interior. These are the next elements of $\mathcal T$. Finally
we consider the negative rays from $P$  and positive rays from $Q$ at time $\eps$. 
By the assumption, the union $\mathcal T$ of all of these rays is an embedded tree in $X$
(see Figure~\ref{fig:collapse:2zeros}).

We now consider a neighborhood $\mathcal T(\delta) = \{ p\in X;\ \mathrm{h}(p,\mathcal T)<\delta \}$ of
$\mathcal T$, where $\mathrm{h}$ is the distance measured in the vertical direction. For $\delta>0$ small enough,
$\mathcal T$ is a retract by deformation of $\mathcal T(\delta)$. We can easily construct $\mathcal T(\delta)$
from 10 Euclidian rectangles whose heights are equal to $\delta$ and widths are equal to $\ell+2\varepsilon$. 

We will now change the flat metric of $\mathcal T(\delta)$ without changing the metric outside of this neighborhood.
Given any $\ell' \in ]0,\ell[$, by varying the points where the rectangles are sewn, we can obtain a new surface
$(X',\omega')$ in $\PrD(2,2)^\mathrm{odd}$ with a saddle connection invariant by the involution whose length is equal to $\ell'$.
Note that the surfaces obtained from this construction belong the same leaf of the kernel foliation as $(X,\omega)$.

When $\ell'=0$, we get a new closed surface $(X_0,\omega_0) \in \H(4)$ sharing the same absolute periods as $(X,\omega)$. Moreover
there exists an involution $\tau_0$ on $X_0$ such that $\tau^*_0\omega_0=-\omega_0$. Hence
$(X_0,\omega_0) \in \PrD(4)$. \medskip

\begin{figure}[htb]
 \begin{minipage}[t]{0.3\linewidth}
 \begin{tikzpicture}[scale=0.2]
\foreach \x in {(-15,4), (-15,-4), (-13,0), (-10,2), (-10,-2)} \draw \x -- (-11,0);
\foreach \x in {(-1,4), (-1,-4), (-3,0), (-6,2), (-6,-2)} \draw \x -- (-5,0);
\draw (-11,0) -- (-5,0);
\draw (-8,0) node[above] {$\scriptstyle \ell$} (-10,2) node[above] {$\scriptstyle \eps$} (-6,2) node[above] {$\scriptstyle \eps$} (-10,-2) node[below] {$\scriptstyle \eps$} (-6,-2) node[below] {$\scriptstyle \eps$} (-13,0) node[left] {$\scriptstyle \eps$} (-3,0) node[right] {$\scriptstyle \eps$};
\draw (-15,4) node[above] {$\scriptstyle \ell+\eps$} (-15,-4) node[below] {$\scriptstyle \ell+\eps$} (-1,4) node[above] {$\scriptstyle \ell+\eps$} (-1,-4) node[below] {$\scriptstyle \ell+\eps$};
\end{tikzpicture}
\end{minipage}
 \begin{minipage}[t]{0.65\linewidth}
\begin{tikzpicture}[scale=0.35]
\foreach \x in {(-14,1), (-14,-1), (-8,1), (-8,-1), (-2,1), (-2,-1), (4,1), (4,-1), (10,1), (10,-1)} \filldraw[fill=blue!20!white!80] \x -- +(4,0) -- +(4,1) -- +(0,1) -- cycle;

\foreach \x in {(-13,1), (-13,0), (-7,1), (-7,0), (-1,1), (-1,0)} \filldraw[fill=white] \x circle (4pt);
\foreach \x in {(-11,1), (1,0), (7,1), (7,0), (13,1), (13,0)} \filldraw[fill=black] \x circle (4pt);

\draw (-13.5,0.5) node {$\scriptstyle \eps$} (-12,0.5) node {$\scriptstyle \ell$} (-10.5,0.5) node {$\scriptstyle \eps$};

\draw (-7.5,0.5) node {$\scriptstyle \eps$} (-5.5,0.5) node {$\scriptstyle \ell+\eps$} (-1.5,0.5) node {$\scriptstyle \eps$} (0,0.5) node {$\scriptstyle \ell$} (1.5,0.5) node {$\scriptstyle \eps$} (5.5,0.5) node {$\scriptstyle \ell+\eps$} (7.5,0.5) node {$\scriptstyle \eps$} (11.5,0.5) node {$\scriptstyle \ell+\eps$} (13.5,0.5) node {$\scriptstyle \eps$};
\draw[dashed] (-14,-1.5) -- (14,-1.5);

\def\a{-4}
\foreach \x in {(-14,1+\a), (-14,-1+\a), (-8,1+\a), (-8,-1+\a), (-2,1+\a), (-2,-1+\a), (4,1+\a), (4,-1+\a), (10,1+\a), (10,-1+\a)}
\filldraw[fill=blue!20!white!80] \x -- +(4,+0) -- +(4,1) -- +(0,1) -- cycle;

\foreach \x in {(-12,1+\a), (-12,0+\a), (-6,1+\a), (-6,0+\a), (0,1+\a), (0,0+\a), (6,1+\a), (6,0+\a), (12,1+\a), (12,0+\a)}
\filldraw[fill=gray] \x circle (4pt);
\draw (-13.25,1+\a) node[below] {$\scriptstyle \ell/2+\eps$}; 
\end{tikzpicture}
\end{minipage}
\caption{Collapsing two zeros along a saddle connection invariant by $\tau$.}
\label{fig:collapse:2zeros}
\end{figure}
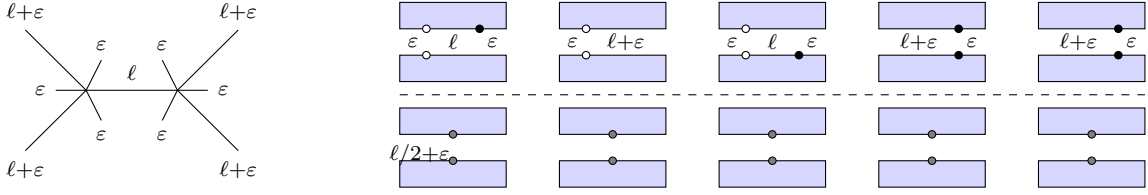

Let us now give the proof of the lemma without the additional condition $({\mathcal C})$. Using $\left(\begin{smallmatrix} 1 & 0 \\ 0 & e^t \end{smallmatrix} \right), \; t >0$, we can assume that any non-horizontal saddle connection has length $> 4\ell$. Set 
$$
\ell_0=\min\{|\s|, \; \s \text{ is a simple horizontal geodesic loop at $P$ or $Q$}\}.
 $$
Choose any $\delta \leq 1/2\min\{\ell, \ell_0\}$, and consider $B(P,\delta):=\{x \in X, \; \dist(x,P) < \delta\}$ and $B(Q,\delta):=\{x \in X, \; \dist(x,Q) < \delta\}$, where $\dist$ is the distance defined by flat metric.  By assumption, $B(P,\delta)$ and $B(Q,\delta)$ are two embedded topological disks in $X$ which are disjoint. Therefore, the surface $(X,\omega)+(-\delta,0)$, which is obtained by moving $P$ by $\delta/2$ to the right, and $Q$ by $\delta/2$ to the left, is well defined. In the new surface, any horizontal saddle connection from $P$ to $Q$ has length reduced by $\delta$, but the lengths of all horizontal geodesic loops are unchanged since they are absolute periods of $\omega$. Note also that if $\s$ is another horizontal  saddle connection joining $P$ to $Q$, then $|\s|-|\s_0|$ is also unchanged. It follows that after finite steps, we can find a surface in the horizontal leaf of $(X,\omega)$ such that $|\s_0| < 1/2\ell_0$ and $|\s_0| < 1/2|\s|$ for any other horizontal saddle connection $\s$ from $P$ to $Q$. We can now apply the above arguments to conclude.

\bigskip

We now turn into the case when $(X,\omega)\in \PrD(1,1,2)$. The construction is similar, we keep
the same convention: $\s_0$ is a positive ray for $R_1$ and a negative ray for $Q$. Note that $\tau(\s_0)$ is a horizontal saddle connection from $Q$ to $R_2$, it is a positive ray for $Q$ and negative ray for $R_2$. We denote
by $\s_{Q,i}^\pm$, $i=1,2,$ the two other positive/negative rays from $Q$, and $\s_{R_1}^+$ (resp. $\s_{R_2}^-$) the other
positive (resp. negative) ray from $R_1$ (resp. from $R_2$).

We again prove the proposition with a slightly stronger assumption  that for any other horizontal
saddle connection $\s$, one has $|\s|>4|\s_0|$.   We will construct a set $\mathcal T$ which is a union of positive/negative rays parametrized by the length to its origin. The first elements of $\mathcal T$ are $\s_0$ and $\tau(\s_0)$.  We then add to $\mathcal T$
\begin{itemize}
\item[$\bullet$]  the rays $\s_{R_1}^+$ and $\s_{R_2}^-$ a time $2\ell+\eps$

\item[$\bullet$] the negative rays from $R_1$ and positive rays from $R_2$ at time $\eps$.

\item[$\bullet$] the positive and negative rays from $Q$ other than $\s_0$ and $\tau(\s_0)$ at time $\ell+\eps$.
\end{itemize}

\noindent with $\eps >0$ small. Now if the ray $\s_{R_1}^+$ intersects
any negative horizontal ray then, by assumption, the associated saddle connection has
length $> 4\ell$. Hence we can choose $\eps > 0$ such that positive ray
$\s_{R_1}^+$ at time $2\ell+\eps$ does not intersect any negative ray from $Q$  at time $\ell+\eps$, nor any negative ray from $R_1$ or $R_2$ at time $2\ell+\eps$.   Similar arguments apply for other positive rays. It follows that  $\mathcal T$ is a tree.

We now consider a neighborhood $\mathcal T(\delta) = \{ p \in X;\ \mathrm{h}(p,\mathcal T)<\delta \}$ of
$\mathcal T$ where $\mathrm{h}$ is the distance measured in the vertical direction. For $\delta>0$ small enough,
$\mathcal T$ is a retract by deformation of $\mathcal T(\delta)$. We can easily construct $\mathcal T(\delta)$
from 10 Euclidian rectangles whose heights are equal to $\delta$ and widths are equal to $2\ell+2\varepsilon$.
As above we can change the flat metric of $\mathcal T(\delta)$ without changing the metric outside of this neighborhood.
The rest of the proof follows the same lines as the case $\kappa=(2,2)^{\rm odd}$.
\end{proof}

\subsection{Twins and non connectedness of Prym eigenform loci when $D=9$}
\label{subsec:connectedness}

In this section we give another elementary proof of the non connectedness of
the loci $\Omega E_9(2,2)^{\rm odd}$ (Theorem~\ref{thm:Dodd:disconnect}).

\begin{proof}[Another proof of Theorem~\ref{thm:Dodd:disconnect}, case $\Omega E_9(2,2)$]
Set $X_0:=\Sur_{(2,2)}(1,1,-1)$ and $X_1:=\Sur_{(2,2)}(1,1,1)$ (see Lemma~\ref{lm:example:D}).
For $i=0,1$, let $\mathcal C_i$ be the connected component of $(X_i,\omega_i)$.

We claim that on any surface in $\mathcal C_0$, any saddle connection which connects the
two zeros of $\omega_0$ has exactly two twins. Since this property is not satisfied by
$(X_1,\omega_1)$ (the longest horizontal saddle connection on $(X_1,\omega_1)$
connects the zeros of $\omega_1$ and has no other twins) this will prove the theorem for $\Omega E_9(2,2)^{\rm odd}$.

By construction the surface $(X_0,\omega_0)$ has three distinct Prym involutions, each of which preserves
exactly one of the tori in the decomposition shown in Figure~\ref{fig:examples:D}.
By Lemma~\ref{lm:2inv:tor:cover} there exists a ramified covering $p : X_0 \rightarrow N_0$ of degree three
(ramified at the zeros of $\omega_0$)  where $N_0$ is a torus. Hence any saddle connection on $X_0$
which connect the zeros of $\omega_0$ has two other twins. For any surface in the kernel foliation leaf or in the
$\GL^+(2,\R)$-orbit of $(X_0,\omega_0)$, this property is clearly preserved (since surfaces still have three
Prym involutions). Thus $(X_1,\omega_1)$ does not belong to the component of $(X_0,\omega_0)$.
\end{proof}

\begin{Remark}
\label{rk:no:twins}
On the surface $\Sur_{(2,2)}(1,1,1)$, the longest horizontal saddle connection (the one that is contained in the boundary components of the bottom cylinder) satisfies Convention~\ref{convention:saddle}, and has no twins. It is not admissible since there are two other saddle connections in the same direction with smaller length. If we move this surface slightly in the kernel foliation leaf to break this parallelism, we will find a twin of this saddle connection (see Figure~\ref{fig:no:twin:not:adm}). This example shows that having a saddle connection satisfying Convention~\ref{convention:saddle} which has no twins does not imply the  existence of an admissible one. On the other hand, we know that $\Sur_{(2,2)}(1,1,1)$ belongs to $\Omega E_9(2,2)^{\rm odd}$ and $\Omega E_9(4)=\ety$, so there exists no admissible saddle connection on $\Sur_{(2,2)}(1,1,1)$.
\begin{figure}[htb]
\begin{minipage}[t]{0.4\linewidth}
\centering
\begin{tikzpicture}[scale=0.5, >=angle 60]
\draw (2,0) -- (5,0) -- (5,4) -- (2,4) -- (2,6) -- (3,6) -- (3,8) -- (2,8)  (1,8) -- (1,6) -- (0,6) -- (0,4) -- (1,4) -- (1,0);
\foreach \x in {(1,0), (1,4), (1,6), (1,8)} \draw[red] \x -- +(1,0);
\foreach \x in {(1,8), (3,8), (1,6), (3,6), (1,4), (5,4), (1,0), (5,0)} \filldraw[fill=white] \x circle(2pt);
\foreach \x in {(2,8), (2,6), (2,4), (2,0), (0,6), (0,4)} \filldraw[fill=black] \x circle(2pt);
\foreach \x in {(1,7), (0,5)} \draw[<->, dashed] \x -- +(2,0); \draw[<->, dashed] (1,2) -- (5,2);
\foreach \x in {(-0.2,4), (-0.2,6)} \draw[<->, dashed] \x -- +(0,2); \draw[<->, dashed] (-0.2,0) -- (-0.2,4);
\draw (-0.2,2) node[left] {$\scriptstyle 2$}; \draw (-0.2,5) node[left] {$\scriptstyle 1$} (-0.2,7) node[left] {$\scriptstyle 1$};
\draw (3,2) node[above] {$\scriptstyle 2$} (1,5) node[above] {$\scriptstyle 1$} (2,7) node[above] {$\scriptstyle 1$};
\draw (3.5,0) node[below] {$\scriptstyle \s_0$};
\end{tikzpicture}
\end{minipage}
\begin{minipage}[t]{0.4\linewidth}
\centering
\begin{tikzpicture}[scale=0.5, >=angle 60]
\draw (2,-1) -- (5,0) -- (5,4) -- (2,3) -- (2,5) -- (3,6) -- (3,8) -- (2,7)  (1,8) -- (1,6) -- (0,5) -- (0,3) -- (1,4) -- (1,0);
\foreach \x in {(1,0), (1,4), (1,6), (1,8)} \draw[red] \x -- +(1,-1);
\draw[dashed] (0,5) -- (3,6);
\foreach \x in {(1,8), (3,8), (1,6), (3,6), (1,4), (5,4), (1,0), (5,0)} \filldraw[fill=white] \x circle(2pt);
\foreach \x in {(2,7), (2,5), (2,3), (2,-1), (0,5), (0,3)} \filldraw[fill=black] \x circle(2pt);
\draw (3.5,-1) node {$\scriptstyle \s_0$} (1.8,6.1) node {$\scriptstyle \s'_0$};
\end{tikzpicture}
\end{minipage}
\caption{On the left: $(X,\omega)=\Sur_{(2,2)}(1,1,1)$, on the right: $(X,\omega)+(0,\eps)$. In $(X,\omega)$, $\s_0$ has no twin, but in $(X,\omega)+(0,\eps)$ it has one.}
\label{fig:no:twin:not:adm}
\end{figure}
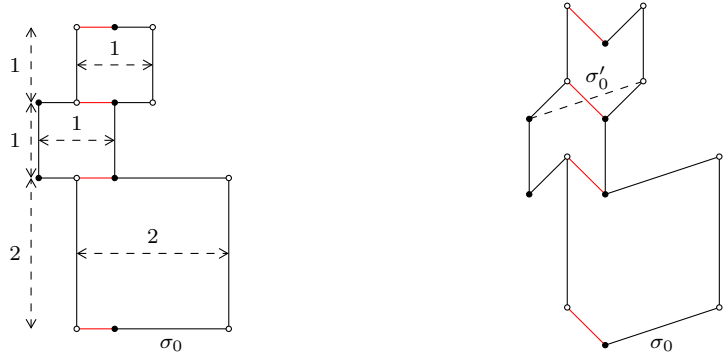

\end{Remark}

\subsection{Twins and triple tori}

The next lemma provides a useful criterion to have triple tori from twin saddle connections
(see Section~\ref{sec:twin} for the definition of triple tori).

\begin{Lemma}
\label{lm:3twins:decomp:3tori}
Let $(X,\omega)$ be a translation surface and let $\s_0$ be a saddle connection on $X$ satisfying
Convention~\ref{convention:saddle}. We assume that $\s_1$ is a twin of $\s_0$ that is not invariant by $\tau$.
\begin{enumerate}
\item If $(X,\omega)\in \Prym(2,2)^\mathrm{odd}$ and $\s_1 \cup \tau(\s_1)$
is separating then the triple of saddle connections $\s_0,\s_1,\tau(\s_1)$ decomposes
$(X,\omega)$ into a triple of flat tori.

\item If $(X,\omega)\in \Prym(1,1,2)$  and $\s_0 \cup \s_1$ is separating then the pairs
$(\s_0,\s_1)$ and $(\tau(\s_0),\tau(\s_1))$ decomposes
$(X,\omega)$ into a triple of flat tori.
\end{enumerate}
\end{Lemma}

\begin{proof}[Proof of Lemma~\ref{lm:3twins:decomp:3tori}]
As usual we assume first that $(X,\omega)\in \Prym(2,2)^{\rm odd}$.
Let $\s_2 = \tau(\s_1)$. We first begin by observing that $(X,\omega)$ is the connected sum of a flat torus
$(X_0,\omega_0)$ with a surface $(X',\omega')\in \H(1,1)$, along $\s_1\cup \s_2$. Indeed the saddle connections
$\s_1$ and $\s_2$ determine a pair of angle $(2\pi,4\pi)$ at $P$ and $Q$. Since $\tau$ permutes $P$ and $Q$, and
preserves the orientation of $X$, it turns out that the angles $2\pi$ at $P$ and the angle
$2\pi$ at $Q$ belong to the same side of $\s_1\cup \s_2$.

As subsurfaces of $X$, $X_0$ and $X'$ have a boundary that consists of the saddle connections $\s_1$ and $\s_2$. We can glue $\s_1$ and $\s_2$ together to obtain two closed surfaces that we continue to denote by $X_0$ and $X'$. We now have on $X_0$ (respectively, on $X'$) a marked geodesic segment $\s$ (respectively, a saddle connection $\s'$) that corresponds to the identification of $\s_1$ and $\s_2$. Note also that $\s_0$ is contained in $X'$.

The involution $\tau$ induces two involutions: $\tau_0$ on $X_0$ and $\tau'$ on $X'$. The involution
$\tau_0$ is uniquely determined by the properties $\tau_0(\omega_0)=-\omega_0$ and $\tau_0$ permutes the endpoints
of $\s$ (namely, $P$ and $Q$). Hence $\tau_0$ is the elliptic involution and has in particular $4$ fixed points:
the midpoint of $\s$ and three fixed points of $\tau$. Since $\tau$ has $4$ fixed points, $\tau'$ has exactly
$2$ fixed points: the midpoint of $\s_0$ ($\s_0$ is invariant by $\tau$) and the midpoint of $\s'$.

Let $\iota$ be the hyperelliptic involution of $X'$. Since $\iota$ has 6 fixed points, we derive $\tau'\neq \iota$. We claim that $\iota(\s_0)=-\s'$.
Indeed, $\iota(\s_0)$ is a saddle connection such that $\omega'(\iota(\s_0))=-\omega'(\s_0)$. Hence
$$
\iota(\s_0)=-\s_0 \qquad \mathrm{or} \qquad \iota(\s_0)=-\s'.
$$
If $\iota(\s_0)=-\s_0$ then $\tau'\circ\iota(\s_0)=\s_0$, hence $\tau'\circ\iota$ is the identity map in the neighborhoods of
$P$ and $Q$. Therefore $\tau'\circ\iota=\id_{X'}$: this is a contradiction since we know that  $\tau'\neq \iota$.
Now the closed curve $\s_0*(-\s')$ is preserved by $\iota$, hence it is separating. Cut $X'$ along this closed curve we obtain two flat tori $(X_1,\omega_1)$ and $(X_2,\omega_2)$. It is not difficult to see that $X_1$ and $X_2$ are exchanged by $\tau'$. By construction, $X$ is the connected sum of $X_0$, $X_1$, and $X_2$ which are glued together along the slits corresponding to $\s_0,\s_1, \s_2$.

The proof for the  case $(X,\omega)\in \Prym(1,1,2)$ is similar, it follows the same lines as the  above discussion.
\end{proof}

\section{Collapsing surfaces to $\PrD(4)$}
\label{sec:collapsing}

The goal of this section is to establish the following theorem, which is a key step in the proof
of Theorem~\ref{thm:main}.
\begin{Theorem}
\label{thm:no:admisSC}
Let $\kappa \in \{(2,2)^{\rm odd},(1,1,2)\}$. Let $\mathcal C$ be a connected component of $\PrD(\kappa)$.
If for every surface $(X,\omega)\in\mathcal C$ there is no admissible saddle connection
on $(X,\omega)$ then $D\in \{9,16\}$. More precisely, under this assumption, $\mathcal C$ contains one of the following surfaces
(see Lemma~\ref{lm:example:D}):
$$
\Sur_\kappa(1,1,-1), \Sur_\kappa(1,1,1)\in \Omega E_9(\kappa), \qquad  \textrm{or}\qquad \Sur_\kappa(1,2,0),\Sur_{(2,2)}(2,1,0) \in \Omega E_{16}(\kappa).
$$
\end{Theorem}

Recall that as an immediate consequence we draw an upper bound on the number of connected components of $\PrD(2,2)^{\rm odd}$
and $\PrD(1,1,2)$ (see Section~\ref{sec:proof}).

\subsection{Strategy of the proof of Theorem~\ref{thm:no:admisSC}}
Let $(X,\omega)$ be a Prym eigenform and let $\s_0$ be a saddle connection satisfying Convention~\ref{convention:saddle}. In view of
Corollary~\ref{cor:twin} we can assume that $\s_0$ has a twin or a double twin,
say $\s_1$, otherwise the theorem is proved. Depending the strata, we will distinguish three cases as follows:
\begin{itemize}
\item $\kappa=(2,2)^{\rm odd}$:
\begin{itemize}
\item[Case A] $\tau(\s_1)\neq -\s_1$ and $\s_1 \cup \tau(\s_1)$ is separating.
\item[Case B] $\tau(\s_1)\neq -\s_1$ and $\s_1 \cup \tau(\s_1)$ is non-separating.
\item[Case C] $\tau(\s_1)=-\s_1$.
\end{itemize}
\item $\kappa=(1,1,2)$:
\begin{itemize}
\item[Case A] $\s_1$ is a twin and $\s_0 \cup \s_1$ is separating.
\item[Case B] $\s_1$ is a twin and $\s_0 \cup \s_1$ is non-separating.
\item[Case C] $\s_1$ is a double twin ($\tau(\s_1)=-\s_1$).
\end{itemize}
\end{itemize}
We will first prove Theorem~\ref{thm:no:admisSC} under the assumption of {\bf Case A}. This case
is simpler since Lemma~\ref{lm:3twins:decomp:3tori} applies: $(X,\omega)$ admits a three tori decomposition.
In the other two cases, by Lemma~\ref{lm:D:square} $D$ is a square.
We then prove that {\bf Case B} and {\bf Case C} can be reduced  to {\bf Case A}: this corresponds, respectively, to
Sections~\ref{proof:case:B} and~\ref{proof:case:C}, respectively.

\subsection{Proof of Theorem~\ref{thm:no:admisSC} under the assumption of {\bf Case A}}
\label{proof:case:A}

From now on we will assume that for any $(X,\omega) \in \mathcal C \subset \PrD(\kappa)$, there exists no admissible  saddle connection.
Let $\s_2$ be the image of $\s_1$ by the Prym involution $\tau$. Thanks to Lemma~\ref{lm:3twins:decomp:3tori}
$(X,\omega)$ admits a three tori decomposition: $X_0$ is preserved and $X_1,X_2$ are exchanged by the Prym involution $\tau$.

\begin{Claim}
\label{lm:3tori:cyl:decomp}
There exists $(X,\omega) \in \mathcal C$ such that the horizontal direction is periodic on the tori $(X_0,X_1,X_2)$.
\end{Claim}
\begin{proof}[Proof of Claim~\ref{lm:3tori:cyl:decomp}]
By moving in the kernel foliation leaf and using $\GL^+(2,\R)$ action, we can assume that the slits $\s_i$ are parallel to a simple closed geodesic in  $(X_0,\omega_0)$ which is horizontal. Since $(X,\omega)$ is completely periodic in the sense of Calta (see~\cite{Lanneau:Manh:cp}), the claim follows.
\end{proof}

%

In the sequel we assume that $(X,\omega)$ is decomposed into three horizontal cylinders, say $C_0,C_1,C_2$, along the saddle connections
$\s_0,\s_1,\s_2$, where $C_0$ is fixed, and $C_1,C_2$ are exchanged by the Prym involution. We let $s=|\sigma_i|$.
We denote by $\ell_i,h_i$ the width and the height of the cylinder $C_i$. Obviously $\ell_1=\ell_2$ and $h_1=h_2$ (see Figure~\ref{fig:3twins:3tori} for the notations).

\subsubsection{Case $\kappa=(2,2)^{\rm odd}$}



\begin{Claim}
\label{clm:3twins:3cyl:heights:22}
One of the following two equalities holds: $h_0=h_1$ or $h_0=2h_1$.
\end{Claim}
\begin{proof}[Proof of Claim~\ref{clm:3twins:3cyl:heights:22}]
Let $\delta$ be a saddle connection in $C_0$ joining $P$ to $Q$ which crosses the core curve of $C_0$ only once. Note that $\delta$ is anti-invariant by $\tau$.  Using $U=\{\left(\begin{smallmatrix} 1 & t \\ 0 & 1 \end{smallmatrix}\right), \; t \in \R\}$, we can assume that $\delta$ is vertical. By assumption, $\delta$ is not admissible. Changing the length of the slits (the length of $\s_i$) if necessary and  using the argument in Lemma~\ref{lm:twin:saddles}, we can assume that $\delta$ has a twin $\delta'$. Remark that $\delta'$  must intersect $C_1\cup C_2$, therefore we have $|\delta'| =mh_1+nh_0$ with
$m\in \Z$, $m\geq 1$.
The condition $|\delta'|=|\delta|$ implies $n=0$. Thus $h_0=mh_1$.

Assume that $h_0 > 2h_1$. Let $\eta_1$ be a geodesic segment in $C_1$ joining $P$ to the midpoint of $\s_0$. Set $\eta_2=\tau(\eta_1)$ and $\eta=\eta_1\cup \eta_2$. Observe that $\eta$ is a saddle connection invariant by $\tau$. Again, by using the subgroup $U$ we assume that $\eta$ is vertical.
Hence $|\eta|=2h_1<h_0$. Clearly any other vertical (upward)  geodesic ray starting from $P$ must intersect $C_0$. Thus, if there exists another vertical saddle connection $\eta'$ joining $P$ to $Q$, we must have $|\eta'|\geq h_0 >2h_1=|\eta|$. Hence $\eta$ is admissible and the claim is proved.
\end{proof}

\begin{Claim}
\label{Claim:4:22}
If $h_0=h_1$ then, either:
\begin{enumerate}
\item $\ell_0=\ell_1$ and $(X,\omega)$ is contained in the same component as $\Sur_{\kappa}(1,1,-1) \in \Omega E_{9}(\kappa)$, or
\item $\ell_0=2\ell_1$ and $(X,\omega)$ is contained in the same component as $\Sur_{\kappa}(1,2,0) \in \Omega E_{16}(\kappa)$.
\end{enumerate}
\end{Claim}
\begin{proof}[Proof of Claim~\ref{Claim:4:22}]
Let $h:=h_0=h_1$. Up to the action of the horocycle flow we assume that $\delta_0$ is a saddle connection in $C_0$ joining $P$ to $Q$
such that $\omega(\delta_0)=(0,h)$. Since $\delta_0$ is not admissible, there exists a vertical saddle connection
$\delta_1$ joining $P$ and $Q$ such that $\omega(\delta_1)= (0,\lambda)$, where $0< \lambda \leq h$.
Actually $\omega(\delta_1)=(0,h)$. Since $\delta_1$ cannot be contained in $C_0$, $\delta_1$ is contained in $C_1$. Thus
$\delta_2=\tau(\delta_1)$ is contained in $C_2$.
Let $\gamma$ be the saddle connection contained in $\overline{C}_1\cup \overline{C}_2$ joining $P$ to $Q$ and passing through the midpoint of
$\s_0$ as shown in Figure~\ref{fig:3cyl:no:admisSC:C1}. By assumption there exists another saddle connection
$\gamma'$ joining $P$ to $Q$ parallel to $\gamma$ such that $|\gamma'| \leq |\gamma|$. We claim that either $\gamma'$ or $\tau(\gamma')$ starts in $C_0$.
This is clear if $\gamma'$ is not invariant by $\tau$. If $\gamma'$ is invariant by $\tau$ and starts from $C_2$ then it must end in $C_1$
(since $\tau(C_2)=C_1$). In particular $\gamma'$ must cross $C_0$ at least once. Hence the vertical coordinate $\omega(\gamma')$ is greater than $2h$.
Therefore $|\gamma'|> |\gamma|$ that is a contradiction.
 \begin{figure}[htb]
 \begin{minipage}[t]{0.4\linewidth}
 \centering
 \begin{tikzpicture}[scale=0.30]
 \draw (2,0) -- (3,0) (6,0) -- (6,4) -- (5,4) -- (5,8) -- (4,8) -- (4,12) -- (3,12) (0,12) -- (0,8) -- (1,8) --  (1,4) -- (2,4) -- (2,0);

 \foreach \x in {(0,12), (1,8), (2,4), (3,0)} \draw[red] \x -- +(3,0);
 \draw[blue, dashed] (0,12) -- (5,4) (1,8) -- (6,0);

 \foreach \x in {(0,12), (4,12), (1,8), (5,8), (2,4), (6,4), (3,0)} \filldraw[fill=white] \x circle (3pt);
 \foreach \x in {(3,12), (0,8), (4,8), (1,4), (5,4), (2,0), (6,0)} \filldraw[fill=black] \x circle (3pt);

 \draw (0,12) node[above] {$\scriptstyle Q$} (3,12) node[above] {$\scriptstyle P$};

 \draw[red] (1.5,12) node[above] {$\scriptstyle \s_2$} (3,8) node[above] {$\scriptstyle \s_0$} (3,4) node[below] {$\scriptstyle \s_1$} (4.5,0) node[below] {$\scriptstyle \s_2$};
 \draw[blue] (1,9.5) node {$\scriptstyle \gamma$} (5.2,2.2) node {$\scriptstyle \gamma'$};

 \draw (0,10) node[left] {$\scriptstyle C_2$} (1,6) node[left] {$\scriptstyle C_1$} (2,2) node[left] {$\scriptstyle C_0$};

 \draw (3,-2) node {$h_0=h_1, \ell_0=\ell_1$};
 \end{tikzpicture}
 \end{minipage}
 \begin{minipage}[t]{0.4\linewidth}
 \centering
 \begin{tikzpicture}[scale=0.30]
 \draw (2,0) -- (7,0) (10,0) -- (10,4) -- (5,4) -- (5,8) -- (4,8) -- (4,12) -- (3,12) (0,12) -- (0,8) -- (1,8) --  (1,4) -- (2,4) -- (2,0);

 \foreach \x in {(0,12), (1,8), (2,4), (7,0)} \draw[red] \x -- +(3,0);
 \draw[blue, dashed] (0,12) -- (5,4) (2,4) -- (4.5,0)  (7.5,4) -- (10,0);

  \foreach \x in {(0,12), (4,12), (1,8), (5,8), (2,4), (10,4), (7,0)} \filldraw[fill=white] \x circle (3pt);
 \foreach \x in {(3,12), (0,8), (4,8), (1,4), (5,4), (2,0), (10,0)} \filldraw[fill=black] \x circle (3pt);

 \draw (0,12) node[above] {$\scriptstyle Q$} (3,12) node[above] {$\scriptstyle P$};

 \draw[red] (1.5,12) node[above] {$\scriptstyle \s_2$} (2,8) node[below] {$\scriptstyle \s_0$} (3.5,4) node[below] {$\scriptstyle \s_1$} (8.5,0) node[below] {$\scriptstyle \s_2$};

  \draw[blue] (1,9.5) node {$\scriptstyle \gamma$} (4.2,1.5) node {$\scriptstyle \gamma'$} (9.5,2) node {$\scriptstyle \gamma'$};

 \draw (0,10) node[left] {$\scriptstyle C_2$} (1,6) node[left] {$\scriptstyle C_1$} (2,2) node[left] {$\scriptstyle C_0$};

  \draw (5,-2) node {$h_0=h_1, \ell_0=2\ell_1$};
 \end{tikzpicture}
 \end{minipage}
 \caption{Claim~\ref{Claim:4:22}: $h_1=h_0$: the surfaces $\Sur_{\kappa}(1,1,-1)$ and
 $\textrm{diag}(1,\frac1{2})\cdot \Sur_{\kappa}(1,2,0)$.}
 \label{fig:3cyl:no:admisSC:C1}
 \end{figure}
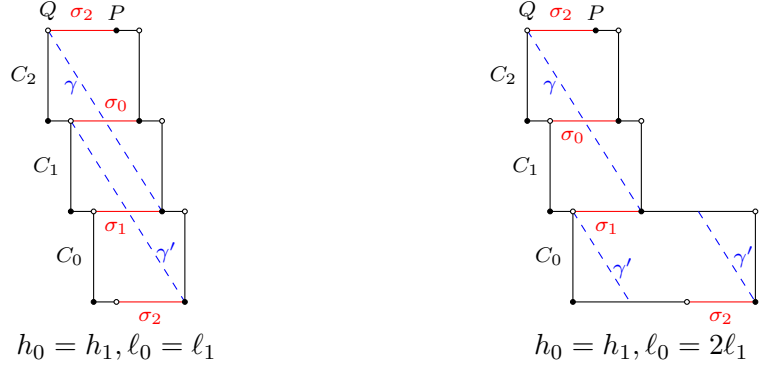

We can now suppose that $\gamma'$ starts in $C_0$. Observe that $\omega(\gamma)=(s-2\ell_1,2h)$. \medskip

If $\gamma'$  is not contained in $\overline{C}_0=X_0$ (Figure~\ref{fig:3cyl:no:admisSC:C1}, left),
$\gamma'$ must end up in $C_1$. Since $|\gamma'| \leq |\gamma|$ elementary calculation shows that $\omega(\gamma')=(s-\ell_0-\ell_1,2h)$.
Now $\gamma$ and $\gamma'$ are parallel, thus $\ell_0=\ell_1$ and $(X,\omega) =\Sur_{\kappa}(1,1,-1)$. \medskip

If $\gamma'$ is contained in $\overline{C}_0$ (Figure~\ref{fig:3cyl:no:admisSC:C1}, right), $\gamma'$ must intersect twice the core curve of $C_0$.
Thus $\omega(\gamma')=(s-\ell_0,2h)$, from which we deduce $\ell_0=2\ell_1$ and $(X,\omega)= \textrm{diag}(1,\frac1{2})\cdot\Sur_{\kappa}(1,2,0)$.
The claim is proved.
\end{proof}

\begin{Claim}
\label{Claim:5:22}
If $h_0=2h_1$ then, either:
\begin{enumerate}
\item $\ell_0=\ell_1$ and  $(X,\omega)$ is contained in the same component as $\Sur_{\kappa}(2,1,0) \in \Omega E_{16}(\kappa)$, or
\item $\ell_0=2\ell_1$ and $(X,\omega)$ is contained in the same component as $\Sur_{\kappa}(1,1,1)\in \Omega E_{9}(\kappa)$.
\end{enumerate}
\end{Claim}

\begin{proof}[Proof of Claim~\ref{Claim:5:22}]
Let $h:=h_1=h_0/2$. Let $\delta_1$ be a geodesic segment, contained in $C_1$, joining the midpoint of $\s_0$ to $P$.
Using the horocycle flow we assume $\delta_1$ to be vertical. Set $\delta_2=\tau(\delta_1)$ and $\delta=\delta_1\cup \delta_2$.
By construction $\delta$ is a saddle connection which is invariant under $\tau$. By assumption,
there exists another vertical saddle connection $\delta'$ joining $P$ to $Q$ such that $|\delta'|\leq |\delta|$. It is easy to see that any other vertical
geodesic ray emanating from $P$ intersects the core curve of $C_0$. Since $\omega(\delta)=(0,2h)$ and $h_0=2h$, $\delta'$ is
contained in $C_0$.

Now let $\gamma$ be the saddle connection in $\overline{C}_1\cup \overline{C}_2$ passing through the midpoint of $\s_0$, joining $P$ to $Q$ such that $\omega(\gamma)=(2\ell_1,2h)$ (see Figure~\ref{fig:3cyl:no:admisSC:C2} below).

\begin{figure}[htb]
 \begin{minipage}[t]{0.4\linewidth}
 \centering
 \begin{tikzpicture}[scale=0.5]
 \draw (1,0) -- (2,0) (4,0) -- (4,3) -- (3,3) -- (4,4.5) -- (5,4.5) -- (6,6) -- (5,6) (3,6) -- (2,4.5) -- (1,4.5) -- (0,3) -- (1,3) -- (1,0);
 \draw[dashed] (3,6) -- (3,3);

 \foreach \x in {(3,6), (2,4.5), (1,3), (2,0)} \draw[red] \x -- +(2,0);
 \draw[blue, dashed] (0,3) -- (6,6) (4,3) -- (1,1.5) (4,1.5) -- (1,0);

 \foreach \x in {(3,6), (6,6), (2,4.5), (5,4.5), (1,3), (4,3), (2,0)} \filldraw[fill=white] \x circle (2pt);
 \foreach \x in {(5,6), (1,4.5), (4,4.5), (0,3), (3,3), (1,0), (4,0)} \filldraw[fill=black] \x circle (2pt);

 \draw (3,6) node[above] {$\scriptstyle Q$} (5,6) node[above] {$\scriptstyle P$};


  \draw[blue] (1.5,4) node {$\scriptstyle \gamma$} (2.5,1.8) node {$\scriptstyle \gamma'$} (3.5,0.7) node {$\scriptstyle \gamma'$};

 \draw (2.5,5.5) node[left] {$\scriptstyle C_2$} (0.5,4) node[left] {$\scriptstyle C_1$} (1,1) node[left] {$\scriptstyle C_0$};

 \draw (3,-1) node {$h_0=2h_1,\ell_0=\ell_1$};
 \end{tikzpicture}
 \end{minipage}
 \begin{minipage}[t]{0.4\linewidth}
\begin{tikzpicture}[scale=0.5]
 \draw (1,0) -- (5,0) (7,0) -- (7,3) -- (3,3) -- (4,4.5) -- (5,4.5) -- (6,6) -- (5,6) (3,6) -- (2,4.5) -- (1,4.5) -- (0,3) -- (1,3) -- (1,0);
 \draw[dashed] (3,6) -- (3,3);

 \foreach \x in {(3,6), (2,4.5), (1,3), (5,0)} \draw[red] \x -- +(2,0);
 \draw[blue, dashed] (0,3) -- (6,6)  (7,3) -- (1,0);

 \foreach \x in {(3,6), (6,6), (2,4.5), (5,4.5), (1,3), (7,3), (5,0)} \filldraw[fill=white] \x circle (2pt);
 \foreach \x in {(5,6), (1,4.5), (4,4.5), (0,3), (3,3), (1,0), (7,0)} \filldraw[fill=black] \x circle (2pt);

 \draw (3,6) node[above] {$\scriptstyle Q$} (5,6) node[above] {$\scriptstyle P$};


  \draw[blue] (1.5,4) node {$\scriptstyle \gamma$} (4.5,1.4) node {$\scriptstyle \gamma'$};

 \draw (2.5,5.5) node[left] {$\scriptstyle C_2$} (0.5,4) node[left] {$\scriptstyle C_1$} (1,1.5) node[left] {$\scriptstyle C_0$};

 \draw (3.5,-1) node {$h_0=2h_1, \ell_0=2\ell_1$};
 \end{tikzpicture}

 \end{minipage}
 \caption{Claim~\ref{Claim:5:22}: $h_0=2h_1$: the surface on the left belongs to component of $\Sur_{\kappa}(2,1,0)$, and on the right belongs to the component of $\Sur_{\kappa}(1,1,1)$.}
  \label{fig:3cyl:no:admisSC:C2}
 \end{figure}
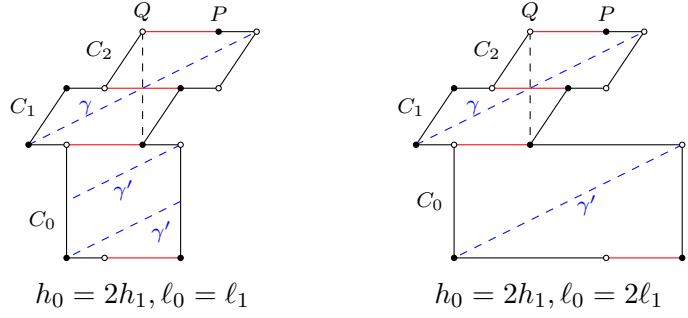
By assumption, there exists a saddle connection $\gamma'$ in the same direction as $\gamma$ such that $|\gamma'|\leq |\gamma|$.
As in Claim~\ref{Claim:4:22} either $\gamma'$ or $\tau(\gamma')$ starts in $C_0$. The proof follows the same lines.
Up to permutation, $\gamma'$ starts in $C_0$. Since $|\gamma'|\leq |\gamma|$ and $\gamma$ is parallel to $\gamma'$,
$\gamma'$ is actually contained in $C_0$. In particular $\omega(\gamma')=(k\ell_0,2h)$ for some $k\in \Z$, $k\geq 1$
and $\omega(\gamma')=\omega(\gamma)$. We draw $2\ell_1=k\ell_0$. Now we claim that the inequality
$$
\ell_0\geq \ell_1
$$
holds. Indeed there exists a horizontal saddle connection $\s'_0$ in $X_0$ such that $\s_1\cup \s'_0$ and $\s_2\cup\s'_0$ are the two
boundaries components of the cylinder $C_0$. Similarly, there exists a pair of horizontal saddle connections $\s'_1, \s'_2$ where $\s'_i$ is contained in $X_i$ such that
$\s'_i\cup\s_0$ is a boundary component of $C_i$. By construction we have $\tau(\s'_0)=-\s'_0$, $\tau(\s'_1)=-\s'_2$, and
$$
\ell_0=|\s_1|+|\s'_0| \qquad \text{  and } \qquad \ell_1=|\s_0| +|\s'_1|=|\s_0|+|\s'_2|.
$$
If $\ell_0<\ell_1$ then $|\s'_0|<|\s'_1|=|\s'_2|$. Hence $\s'_0$ is admissible, contradicting our assumption. \medskip

In conclusion $2\ell_1=k\ell_0 \geq \ell_1$ implies $\ell_0=\ell_1$, or $\ell_0=2\ell_1$. The corresponding surfaces are represented in Figure~\ref{fig:3cyl:no:admisSC:C2}. It is not hard to check that those two surfaces
belong to the same connected component that $\Sur_{\kappa}(2,1,0)$ and $\Sur_{\kappa}(1,1,1)$, respectively.  The claim is proved.
\end{proof}
\subsubsection{Case $\kappa=(1,1,2)$}


%

\begin{Claim}
\label{clm:3cyl:heights:112}
One of the following two equalities holds: $h_0=h_1$, or $h_0=2h_1$.
\end{Claim}
\begin{proof}[Proof of Claim~\ref{clm:3cyl:heights:112}]
The proof of this claim follows the same lines as the proof of Claim~\ref{clm:3twins:3cyl:heights:22}.
\end{proof}

\begin{Claim}
\label{Claim:4:112}
If $h_0=h_1$ then, either
\begin{enumerate}
\item $\ell_0=\ell_1$, and $(X,\omega)$  is contained in the same component as $\Sur_\kappa(1,1,-1) \in \Omega E_9(\kappa)$, or
\item $\ell_0=2\ell_1$, and $(X,\omega)$  is contained in the same component as $\Sur_\kappa(1,2,0) \in \Omega E_{16}(\kappa)$.
\end{enumerate}
\end{Claim}
\begin{proof}[Proof of Claim~\ref{Claim:4:112}]
Set $h:=h_0=h_1$. We first consider a saddle connection $\delta$ contained in $C_0$, joining $R_1$ to $Q$ and intersecting
the core curve of $C_0$ only once. As usual we assume $\delta$ to be vertical (hence $\omega(\delta)=(0,h)$).
By assumption $\delta$ has a twin or a double-twin $\delta_1$ (necessarily $\delta_1$ starts in $C_1$).
Clearly $\delta_1$ is a twin: otherwise it must end in $C_2$, hence it must cross the core curve of $C_0$ at least once.
In particular its length satisfies $|\delta_1|\geq  3h> 2|\delta|$ that is a contradiction.

Let $\gamma$ be the saddle connection contained in $C_1$ and joining $R_1$ to $Q$, as shown in Figure~\ref{fig:no:admisSC:112:a} below.

\begin{figure}[htb]
\begin{minipage}[t]{0.4\linewidth}
\centering
\begin{tikzpicture}[scale=0.6]
\foreach \x  in {(0,5), (0,2.5), (1.5,0)} {\draw \x rectangle +(2,2); \draw[dashed] \x -- +(2,2);}
\foreach \x in {(0,2.5), (0.5,4.5), (1.5,0), (2,2.5), (2,2), (3.5,0)} \filldraw[fill=black] \x circle (2pt);
\foreach \x in {(0,7), (1,2.5), (1.5,5), (1.5,4.5), (2,7)} \filldraw[fill=white] \x circle (2pt);
\foreach \x in {(0,5), (0,4.5), (0.5,7), (1.5,2.5), (1.5,2), (2,5), (2,4.5), (3,0), (3.5,2)} {\filldraw[fill=white] \x circle (3pt); \draw \x +(45:3pt) -- +(225:3pt) +(135:3pt) -- +(315:3pt);}

\draw (-0.5,6) node {\tiny $C_2$} (-0.5,3.5) node {\tiny $C_0$} (1,1) node {\tiny $C_1$} (2,7) node[right] {\tiny $R_2$} (2,5) node[right] {\tiny $Q$} (3.5,2) node[right] {\tiny $Q$} (3.5,0) node[right] {\tiny $R_1$};

\draw (0.8,6) node[above] {\tiny $\tau(\gamma)$} (1,3.5) node[above] {\tiny $\gamma'$} (2.5,1)  node[above] {\tiny $\gamma$};

\draw (2.3,3.5) node {\tiny $\delta$};
\draw (4,1) node {\tiny $\delta_1$};
\draw (1.5,-1) node {$\ell_0=\ell_1$};
\end{tikzpicture}
\end{minipage}
\begin{minipage}[t]{0.4\linewidth}
\centering
\begin{tikzpicture}[scale=0.6]
\draw (0,2.5) rectangle (4,4.5); \draw  (2,5) rectangle (4,7); \draw (3.5,0) rectangle (5.5,2);
\foreach \x in {(0,2.5), (1.5,2.5), (2,5), (3.5,0)}  \draw[dashed] \x -- +(2,2);

\foreach \x in {(0,2.5), (0.5,4.5), (3.5,0), (4,2.5), (4,2), (5.5,0)} \filldraw[fill=black] \x circle (2pt);
\foreach \x in {(2,7), (3,2.5), (3.5,5), (3.5,4.5), (4,7)} \filldraw[fill=white] \x circle (2pt);
\foreach \x in {(0,4.5), (2,5),  (2.5,7), (3.5,2.5), (3.5,2), (4,5), (4,4.5), (5,0), (5.5,2)} {\filldraw[fill=white] \x circle (3pt); \draw \x +(45:3pt) -- +(225:3pt) +(135:3pt) -- +(315:3pt);}

\draw (1.5,6) node {\tiny $C_2$} (-0.5,3.5) node {\tiny $C_0$} (3,1) node {\tiny $C_1$} (4,7) node[right] {\tiny $R_2$} (4,5) node[right] {\tiny $Q$} (5.5,2) node[right] {\tiny $Q$} (5.5,0) node[right] {\tiny $R_1$};

\draw (4.3,3.5) node {\tiny $\delta$};
\draw (6,1) node {\tiny $\delta_1$};
\draw (2.8,6) node[above] {\tiny $\tau(\gamma)$} (1,3.5) node[above] {\tiny $\gamma'$} (4.5,1)  node[above] {\tiny $\gamma$};
\draw (2.5,-1) node {$\ell_0=2\ell_1$};
\end{tikzpicture}
\end{minipage}

\caption{Claim~\ref{Claim:4:112}: $h_0=h_1$: the surfaces  $\Sur_\kappa(1,1,-1)$ and $\textrm{diag}(1,\frac1{2})\cdot \Sur_\kappa(1,2,0)$.}
\label{fig:no:admisSC:112:a}
\end{figure}
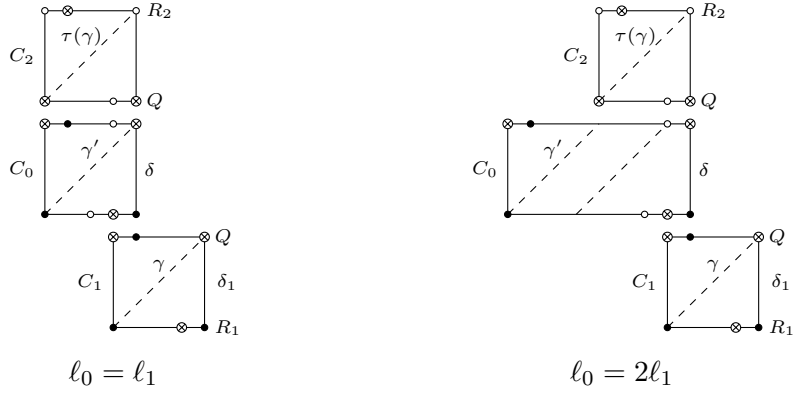

By assumption $\omega(\gamma)=(\ell_1,h)$. Now there exists another saddle connection $\gamma'$ starting from $R_1$ in the same direction
as $\gamma$.  Observe that $\gamma'$ must start in $C_0$. Using Lemma~\ref{lm:twin:saddles}, we can assume that $\gamma'$ is either a twin or a double-twin of $\gamma$. \medskip

If $\gamma'$ is a twin of $\gamma$ then $\omega(\gamma')=(\ell_1,h)$. Hence $\ell_0=\ell_1$: $(X,\omega)=\Sur_\kappa(1,1,-1)$.

If $\gamma'$ is a double-twin of $\gamma$ then $\omega(\gamma')=(2\ell_1,2h)$. Hence $\gamma'$ crosses twice the core curves of
$C_0$ implying that $\ell_0=2\ell_1$: $(X,\omega)=\textrm{diag}(1,\frac1{2})\cdot\Sur_\kappa(1,2,0)$. This proves the claim.
\end{proof}

\begin{Claim}
\label{Claim:5:112}
If $h_0=2h_1$ then $\ell_0=2\ell_1$. In addition $(X,\omega)$ belongs to the same connected component as $\Sur_\kappa(1,1,1) \in \Omega E_9(\kappa)$.
\end{Claim}
\begin{proof}
Set $h=h_1$. Let $\delta_1$ be a saddle connection contained in $C_1$ joining $R_1$ to $Q$ and intersecting the core curve of $C_1$ only once.
We can suppose that $\delta_1$ is vertical. By assumption, there exists another vertical saddle connection $\delta$ starting from $R_1$. Observe that $\delta$ must intersect the core curve of  $C_0$, thus we have $|\delta| \geq h_0 =2|\delta_1|$. By assumption, $\delta$ must be a double-twin of $\delta_1$, which means that $\delta$  joins $R_1$ to $R_2$ and is contained in $C_0$.

Now, let $\gamma$ be the saddle connection in $C_1$ joining $R_1$ to $Q$ as shown in Figure~\ref{fig:no:admisSC:112:b}.

\begin{figure}[htb]
\begin{minipage}[t]{0.4\linewidth}
\centering
\begin{tikzpicture}[scale=0.5]
\foreach \x in {(0,7), (5,0)} {\draw \x rectangle +(2,2); \draw[dashed] \x -- +(2,2);}
\draw (1.5,2.5) rectangle (5.5,6.5); \draw[dashed] (1.5,2.5) -- (5.5,6.5);

\foreach \x in {(1.5,2.5), (2.5,6.5), (5,0), (5.5,2.5), (5.5,2), (7,0)} \filldraw[fill=black] \x circle (2pt);
\foreach \x in {(0,9), (1.5,7), (1.5,6.5), (2,9), (4.5,2.5), (5.5,6.5)} \filldraw[fill=white] \x circle (2pt);
\foreach \x in {(0,7), (0.5,9), (2,7), (2,6.5), (5,2.5), (5,2), (6.5,0), (7,2)} {\filldraw[fill=white] \x circle (3pt); \draw \x +(45:3pt) -- +(225:3pt) +(135:3pt) -- +(315:3pt);}

\draw (-0.5,8) node {\tiny $C_2$} (1,4.5) node {\tiny $C_0$} (4.5,1) node {\tiny $C_1$} (2,9) node[right] {\tiny $R_2$} (2,7) node[right] {\tiny $Q$} (5.5,6.5) node[right] {\tiny $R_2$} (5.5,2.5) node[right] {\tiny $R_1$} (7,2) node[right] {\tiny $Q$} (7,0) node[right] {\tiny $R_1$};

\draw (0.8,8) node[above] {\tiny $\tau(\gamma)$} (3.5,4.5) node[above] {\tiny $\gamma'$} (6,1)  node[above] {\tiny $\gamma$};
\end{tikzpicture}
\end{minipage}
\caption{Claim~\ref{Claim:5:112}: $h_0=2h_1$: the surface  $\Sur_\kappa(1,1,1)$}
\label{fig:no:admisSC:112:b}
\end{figure}
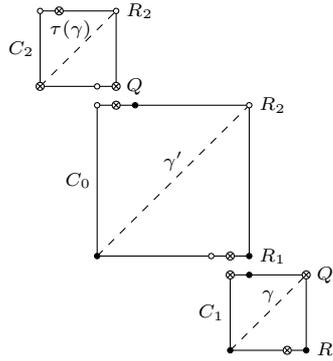

By assumption, there exists a saddle connection $\gamma'$ starting from $R_1$ and parallel to $\gamma$. The same argument as above shows
that $\gamma'$ is a double-twin of $\gamma$ and is contained in $C_0$. It follows that $\ell_0=2\ell_1$. Thus
$(X,\omega)$ belongs to the  component of $\Sur_\kappa(1,1,1)$. The claim is proved.
\end{proof}

\subsection{Reduction from Case B to Case A}
\label{proof:case:B}
Let $(X,\omega) \in \PrD(\kappa)$ and let $\s_0$ be a saddle connection in $X$ satisfying Convention~\ref{convention:saddle}. We suppose that $\s_0$ has a twin $\s_1$. Moreover, if $\kappa=(2,2)^{\rm odd}$ we assume that  $\s_2:=-\tau(\s_1)\neq \s_1$ and $\s_1 \cup \s_2$ is non-separating, and  if $\kappa=(1,1,2)$ we assume that $\s_0,\s_1$ is non-separating. Our aim is to show that there exists in the component of $(X,\omega)$ a surface having a  family of homologous saddle connections satisfying Case A (this is Lemma~\ref{lm:twinSC:B-A:22}). We first show

\begin{Lemma}
\label{lm:D:square}
Let $(X,\omega) \in \Omega E_D(\kappa)$, where $\kappa\in \{(2,2)^{\rm odd}, (1,1,2)\}$. If there exists $c \in H_1(X,\Z)^{-}$ satisfying $c\neq 0$ and $\omega(c)=0$ then $D$ is a square.
In particular, up to rescaling by  $\GL^+(2,\R)$, all the absolute periods of $\omega$ belong to $\Q+\imath\Q$.
\end{Lemma}
\begin{proof}[Proof of Lemma~\ref{lm:D:square}]
We can assume that $c$ is primitive in $H_1(X,\Z)$, that is for any $n \in \N, n>1, \cfrac{1}{n}c \not\in H_1(X,\Z)$. Pick a symplectic basis $(\alpha_1,\beta_1,\alpha_2,\beta_2)$ of $H_1(X,\Z)^-$ with $\beta_2=c$. Set $\mu_i=\langle \alpha_i,\beta_i \rangle$, where $\langle,\rangle$ is the intersection form of $H_1(X,\Z)$, and
$\omega(\alpha_1)=x_1+\imath y_1, \omega(\beta_1)=z_1 +\imath t_1, \omega(\alpha_2)=x_2+\imath y_2$.
Since
$$
\mathrm{Area}(X,\omega)=\mu_1\det\left(\begin{array}{rr}
x_1 & z_1\\
y_1 & t_1\\
\end{array} \right) +\mu_2\det\left(\begin{array}{rr}
x_2 & 0\\
y_2 & 0\\
\end{array}\right) = \mu_1\det\left(\begin{array}{rr}
x_1 & z_1\\
y_1 & t_1\\
\end{array} \right) > 0
$$
\noindent it follows that $(x_1+\imath y_1,z_1+\imath t_1)$ is a basis of $\R^2$. Using $\GL^+(2,\R)$ we can assume that
$(x_1,y_1)=(1,0)$ and $(z_1,t_1)=(0,1)$. By~\cite[Proposition 4.2]{LanNg11} there exists a unique
generator $T$ of $\Ord_D$ which is given in the basis $(\alpha_1,\beta_1,\alpha_2,\beta_2)$ by a matrix of the
form
$$
T=\left(\begin{array}{cc}
e\Id_2 & \left(\begin{smallmatrix} a & b \\ c & d \\ \end{smallmatrix}\right)\\
\frac{\mu_1}{\mu_2}\left(\begin{smallmatrix} d & -b \\ -c & a \\ \end{smallmatrix}\right) & 0 \\
\end{array}\right)
$$
with $(a,b,c,d,e)\in \Z^5$ such that $T^*\omega=\lambda\omega$ and $\lambda>0$. Observe that the
discriminant $D$ satisfies $D=e^2-4 \frac{\mu_1}{\mu_2}(bc-ad)$. Since $\mathrm{Re}(\omega)=(1,0,x_2,0)$
and $\mathrm{Im}(\omega)=(0,1,y_2,0)$ in the above coordinates,
direct computations show that $b=d=0$. Hence $D=e^2$ and $(x_2,y_2) \in \Q^2$. The lemma is proved.
\end{proof}

\begin{Lemma}
\label{lm:twinSC:B-A:22}
Assume that $\kappa=(2,2)^{\rm odd}$. Then there exists in the component of $(X,\omega)$ a surface having a
triple of homologous saddle connections $\gamma_0,\gamma_1,\gamma_2$, where $\gamma_0$ is invariant, $\gamma_1$ and $\gamma_2$ are exchanged by the involution.
\end{Lemma}

\begin{proof}
Let $c_0$,  $c_1$,  and $c_2$ denote the simple closed curves $\s_1*(-\s_2), \s_0*(-\s_1)$ and $\s_0*(-\s_2)$ respectively. Note that  we have $c_0=c_2-c_1$ and $\tau(c_1)=-c_2$. By assumption $0\neq c_0  \in H_1(X,\Z)$. If $0=c_1 \in H_1(X,\Z)$ then $c_2=-\tau(c_1)=0 \in H_1(X,\Z)$, which implies that $c_0=0 \in H_1(X,\Z)$. Thus we can conclude that all of the curves $c_0,c_1,c_2$ are non-separating.

Cut $X$ along $\s_0,\s_1,\s_2$, we obtain a connected  surface whose boundary has three components corresponding to $c_0,c_1,c_2$. Gluing the pair of geodesic segments in each boundary component together, we get a closed translation surface $(X',\omega')$ with three marked geodesic segments. Since the angle between two consecutive twin saddle connections is $2\pi$, we derive that $\omega'$ has no zeros, thus $(X',\omega')$ must be a torus. We denote the geodesic segments in $X'$ corresponding to $c_0,c_1,c_2$ by $c'_0,c'_1,c'_2$ respectively. The involution $\tau$ of $X$  induces an involution $\tau'$ on $X'$, which leaves $c'_0$ invariant and exchanges $c'_1$ and $c'_2$.  Let $P'_i$ and $Q'_i$, $i=0,1,2$, denote the endpoints of $c'_i$, where $P'_i$ (respectively, $Q'_i$) corresponds to $P$ (respectively, to $Q$).

\begin{figure}[htb]
\centering
\begin{tikzpicture}[scale=0.4]
\draw (-12,3) -- +(10,0) -- +(10,-6) -- +(0,-6) -- cycle;
\draw (-12,1) -- (-2,1) (-12,-1) -- (-2,-1);
\foreach \x in {(-9,1), (-7,1), (-5,1)} \draw \x +(0.1,0) -- +(0.1,-2) +(-0.1,0) -- +(-0.1,-2);

\draw (2,-3) -- (8,-3) -- (8,-1) -- (10,-1) -- (10,-3) -- (12,-3) -- (12,3) -- (10,3) -- (10,5) -- (8,5) -- (8,3) -- (2,3) -- cycle;
\draw (2,-1) -- (8,-1) (10,-1) -- (12,-1) (8,3) -- (10,3);

\foreach \x in {(2,3), (8,3), (10,3), (12,3)} \draw \x -- +(0,-4);

\foreach \x in {(-9,1), (-7,1), (-5,1), (2,-1), (8,-1), (10,-1), (8,5), (10,5), (12,-1)} \filldraw[fill=white] \x circle (4pt);
\foreach \x in {(-9,-1), (-7,-1), (-5,-1), (2,3), (2,-3), (8,3), (8,-3), (10,3), (10,-3), (12,3), (12,-3)} \filldraw[fill=black] \x circle (4pt);

\draw (-9,1) node[above] {$\scriptstyle P'_1$} (-7,1) node[above] {$\scriptstyle P'_0$} (-5,1) node[above] {$\scriptstyle P'_2$} (-9,-1) node[below] {$\scriptstyle Q'_1$} (-7,-1) node[below] {$\scriptstyle Q'_0$} (-5,-1) node[below] {$\scriptstyle Q'_2$};

\draw (-12,1) node[left] {$\scriptstyle \alpha^+$} (-12,-1) node[left] {$\scriptstyle \alpha^-$};

\draw (2,1) node[left] {$\scriptstyle \gamma_1$}  (8,1) node[left] {$\scriptstyle \gamma_2$} (10,1) node[left] {$\scriptstyle \gamma_0$} (12,1) node[right] {$\scriptstyle \gamma_1$};

\draw (2,-2) node[left] {\tiny $\s_0$} (8,-2) node[left] {\tiny $\s_0$} (7.6,4) node {\tiny $\s_1$} (10.4,4) node {\tiny $\s_1$} (9.6,-2) node {\tiny $\s_2$} (12.4,-2) node {\tiny $\s_2$};

\draw (-7,-4) node {$(X',\omega')$} (7,-4) node {$(X,\omega)$};

\end{tikzpicture}

\caption{Cylinders decomposition in direction of $\alpha^{\pm}$.}
\label{fig:twinsSC:BC}
\end{figure}
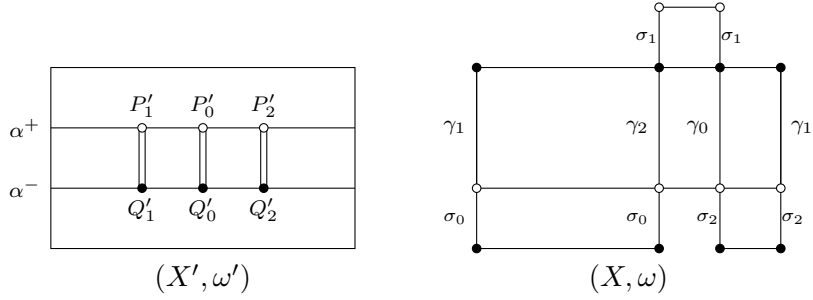

Observe that as $(X,\omega)$ moves in the leaf of the kernel foliation, the surface $(X',\omega')$ is the same, only the segments $c'_i$ vary. Therefore we can assume that $(X',\omega')$ is the standard torus $\C/(\Z\oplus\imath\Z)$, and the length of $c'_0$ is small. Let $\delta_1$ (respectively, $\eta_1$) denote the geodesic segment of minimal length from  $P'_0$ to $P'_1$ (respectively, from $Q'_0$ to $Q'_1$). Note that as $(X,\omega)$ moves in  the kernel foliation leaf, $\omega'(\delta_1)$ and $\omega'(\eta_1)$ are invariant. Therefore, we can assume that $\delta_1$ and $c'_0$ are not parallel.

Since $c'_0$ and $c'_1$ are parallel and have the same length, we see that $c'_0\cup\delta_1\cup c'_1\cup\eta_1$ is the boundary of an embedded parallelogram in $X'$.  It follows in particular that $\delta_1$ and $\eta_1$  are parallel and have the same length.  Let $\delta_2=\tau'(\eta_1)$ and $\eta_2=\tau'(\delta_1)$. We have $\delta=\delta_1\cup \delta_2$ is a geodesic segment joining $P'_1$ to $P'_2$, and $\eta=\eta_1\cup \eta_2$ is a geodesic segment joining $Q'_1$ to $Q'_2$.

Since $\omega(c_0)=0$, By Lemma~\ref{lm:D:square}, for any $c \in H_1(X,\Z)$, we have $\omega(c)\in \Q+\imath\Q$. Therefore  $\omega'(\delta)=\omega'(\eta) \in \Q+\imath\Q$. Since $X'$ is the standard torus, there exists a pair of parallel simple closed geodesics $\alpha^+, \alpha^-$ of $X'$ such that $\delta \subset \alpha^+$, and $\eta \subset \alpha^-$ (see Figure~\ref{fig:twinsSC:BC}).  When $|c'_i|$ is small enough and non-parallel to $\alpha^{\pm}$, the geodesics $\alpha^+$ and $\alpha^-$ cut $X'$ into two cylinders, one of which contains all the segments $c'_0,c'_1,c'_2$.

Recall that $(X,\omega)$ is obtained from $(X',\omega')$ by slitting along $c'_0,c'_1,c'_2$, and regluing the geodesic segments in the boundary. By construction, we see that $(X,\omega)$ admits a decomposition into four cylinders in the direction of $\alpha^\pm$ as shown in Figure~\ref{fig:twinsSC:BC}. Let $C_0$ denote the largest cylinder in this decomposition, then it is easy to see that there exist in $C_0$ three homologous saddle connections $\gamma_0, \gamma_1,\gamma_2$ such that $\tau(\gamma_0)=-\gamma_0$, $\tau(\gamma_1)=-\gamma_2$, and $\gamma_1\cup\gamma_2$ is a separating curve as desired.
\end{proof}

\begin{Lemma}
\label{lm:twinSC:B-A:112}
Assume that $\kappa=(1,1,2)$ and $\s_0$ has a twin $\s_1$ such that the curve $\s_0*(-\s_1)$ is non-separating. Then there exists in the component of $(X,\omega)$ a surface having two pairs of homologous saddle connections $(\s'_1,\s''_1)$ and $(\s'_2,\s''_2)$, where $\s'_i$ and $\s''_i$ join the simple zero $R_i$ to the double zero $Q$, and $\{\s'_2,\s''_2\}=\tau(\{\s'_1,\s''_1\}$).
\end{Lemma}
\begin{proof}
We will use similar ideas to the proof of Lemma~\ref{lm:twinSC:B-A:22}. Let $c_1=\s_0*(-\s_1)$ and $c_2=\tau(c_1)$. By the cutting-gluing construction along $c_1$ and $c_2$ (using the assumption that $c_1$ is non-separating), we get a flat torus $(X',\omega')$ with three marked geodesic segments $c'_1,c'_2, c''$ such that $\omega'(c'_1)=\omega'(c'_2)=1/2\omega'(c'')$ (see Figure~\ref{fig:twinSC:B-A:112}). For $i=1,2$, we denote the endpoints of $c'_i$ by $R'_i$ and $Q'_i$ so that $R'_i$ corresponds to $R_i$ and $Q'_i$ corresponds to $Q$. We denote the endpoints of  $c''$ by $R''_1$ and $R''_2$ such that $R''_i$ corresponds to $R_i$. The midpoint of $c''$ corresponds to $Q$, we denote it by $Q''$. We denote the subsegment of $c''$ between $Q''$ and $R''_i$  by $c''_i$. The Prym involution of $X$ gives rise to an involution $\tau'$ of $X'$ which satisfies $\tau'(c'_1)=-c'_2, \tau'(c''_1)=-c''_2$.
\begin{figure}[htb]
\centering
\begin{tikzpicture}[scale=0.7]
\draw (-3,-2) rectangle (3,2);
\foreach \x in {(-1.5,0), (0,0)} {\draw \x -- +(-0.5,2) +(-0.5,-2) -- +(-0.75,-1); \draw \x ++(0,0.05) -- +(-0.75,-1); \draw \x ++(0,-0.05) -- +(-0.75,-1);}
\foreach \x in {(1.5,0), (0,0)} {\draw \x -- +(0.5,-2) +(0.5,2) -- +(0.75,1); \draw\x ++(0,0.05) -- +(0.75,1); \draw \x ++(0,-0.05) -- +(0.75,1);}

\foreach \x in {(-1.5,0), (0,0), (1.5,0)} {\filldraw[fill=white] \x circle (3pt); \draw \x +(45:3pt) -- +(225:3pt) +(135:3pt) -- +(315:3pt);}
\foreach \x in {(-2.25,-1), (-0.75,-1)} \filldraw[fill=white] \x circle (2pt);
\foreach \x in {(2.25,1), (0.75,1)} \filldraw[fill=black] \x circle (2pt);

\draw (-2,2) node[above] {\tiny $\s'_1$} (-2,-2) node[below] {\tiny $\s'_1$} (-0.5,2) node[above] {\tiny $\s''_1$} (-0.5,-2) node[below] {\tiny $\s''_1$} (0.5,2) node[above] {\tiny $\s''_2$} (0.5,-2) node[below] {\tiny $\s''_2$} (2,2) node[above] {\tiny $\s'_2$} (2,-2) node[below] {\tiny $\s'_2$};

\draw (-1, 0) node {\tiny $Q'_1$} (-2.6, -1) node {\tiny $R'_1$} (0.5,0) node {\tiny $Q''$} (-1.1,-1) node {\tiny $R''_1$} (1.1,1) node {\tiny $R''_2$}  (2.6,1) node {\tiny $R'_2$} (2,0) node {\tiny $Q'_2$};
\end{tikzpicture}
\caption{}
\label{fig:twinSC:B-A:112}
\end{figure}

As $(X,\omega)$ moves in the leaf of the kernel foliation, the surface $(X',\omega')$ remains the same, only the segments $c'_1,c'_2,c''$ vary. Therefore we can assume that  $c'_1,c'_2,c''$ are contained in three distinct parallel simple closed geodesics of $X'$. Changing the direction of $c''$ slightly, we see that there exist geodesic segments $\s'_i$ from $Q'_i$ to $R'_i$, and $\s''_i$ from $Q''$ to $R''_i, \ i=1,2$, (see Figure~\ref{fig:twinSC:B-A:112})  such that
\begin{itemize}
 \item[$\bullet$] $\tau'(\s'_1)=-\s'_2$ and $\tau'(\s''_1)=-\s''_2$,
 \item[$\bullet$] $\s'_i*c'_i$ and $\s''_i*c''_i$ are homologous.
 \end{itemize}
Reconstruct $(X,\omega)$ from  $(X',\omega')$ we see that $\s'_i$ and $\s''_i$ are homologous saddle connections, and the pairs $(\s'_1,\s''_1)$ and $(\s'_2,\s''_2)$ have the desired properties.
\end{proof}

\subsection{Reduction from Case C to Case A}
\label{proof:case:C}
Let $(X,\omega)$ be a Prym eigenform in $\Prym(2,2)^{\rm odd} \sqcup \Prym(1,1,2)$, and  let $\s_0$ be a saddle connection on $X$ satisfying Convention~\ref{convention:saddle}.
If $(X,\omega) \in \Prym(2,2)^{\rm odd}$ we suppose that $\s_0$ has a twin $\s_1$ which is also invariant by  the Prym involution $\tau$, and if $(X,\omega)\in \Prym(1,1,2)$ we suppose $\s_0$ has a double twin $\s_1$ (which is invariant by $\tau$). Our aim is to show that there exists in the component of $(X,\omega)$ a surface having a family of saddle connections satisfying Case A for both $\kappa=(2,2)^{\rm odd}$ and  $\kappa=(1,1,2)$. We first show
\begin{Lemma}
\label{lm:cut}
Define
\begin{itemize}
 \item[$\bullet$] $c= \s_0*(-\s_1)$, if $(X,\omega) \in \Prym(2,2)^{\rm odd}$,
 \item[$\bullet$] $c=\s_0*\tau(s_0)*(-\s_1)$, if $(X,\omega) \in \Prym(1,1,2)$.
\end{itemize}
Then in both cases, we have $c \neq 0 \in H_1(X,\Z)^{-}$, and $\omega(c)=0$.
\end{Lemma}
\begin{proof}
From the definition of $c$,  we have $\tau(c)=-c$, hence $c \in H_1(X,\Z)^{-}$. It is also clear that $\omega(c)=0$.  All we need to show is that $c\neq 0 \in H_1(X,\Z)$.

We first  consider the case $(X,\omega) \in \Prym(2,2)^{\rm odd}$. Remark that the pair of angles at $P$ and $Q$ determined  $\s_0$ and $\s_1$ is $(2\pi,4\pi)$.
Since $\tau(\s_0)=-\s_0$ and $\tau(\s_1)=-\s_1$ we see that the angle $2\pi$ at $P$ and the angle $4\pi$ at $Q$ belong to the same side of $c$,
and vice versa. Cutting $X$ along $c$ we get a surface whose boundary has two components, each of which is a union of two geodesic segments
corresponding to $\s_0$ and $\s_1$. Since $\s_0$ and $\s_1$ are twins, the two segments in each component has the same length, therefore we
can glue them together to get a closed (possibly disconnected) translation surface $(X',\omega')$ with two marked geodesic segments $\eta_1, \eta_2$.
If the new surface is disconnected, then each component is a translation surface with only one singularity of angle $4\pi$. Since such a surface
does not exist, we conclude that $X'$ is connected, and hence $c \neq 0$ in $H_1(X,\Z)$.

For the case $(X,\omega)\in \Prym(1,1,2)$, by a similar  construction, that is cutting along $c$, then closing the boundary components of the new surface (by gluing the path corresponding to $\s_0\cup\tau(\s_0)$ and the segment corresponding to $\s_1$), we also get a translation surface $(X',\omega')$ having two singularities with cone angle $4\pi$. The same argument as above shows that this surface belongs to $\H(1,1)$, therefore $c\neq 0 \in H_1(X,\Z)$.
\end{proof}

\begin{Lemma}
\label{lm:twinsSC:CA:22}
Let $(X,\omega)\in \Prym(2,2)^{\rm odd}$ be a Prym eigenform having a twin $\s_1$ of $\s_0$ that is invariant by $\tau$. Then
one can find in the connected component of $(X,\omega)$ another surface having a triple of
homologous saddle connections.
\end{Lemma}

\begin{proof}[Proof of Lemma~\ref{lm:twinsSC:CA:22}]
Cutting $X$ along $c=\s_0\cup\s_1$ and gluing the two segments of each boundary component together, we get a closed
translation surface $(X',\omega')\in\H(1,1)$ (see the proof of Lemma~\ref{lm:cut}). By construction there exist on $X'$ a pair of
disjoint geodesic segments $\eta_1,\eta_2$ such that $\omega'(\eta_1)=\omega'(\eta_2)=\omega(\s_0)$ and $\eta_i$ joins a
zero of $\omega'$ to a regular point. Let $(P_i, Q_i)_{i=1,2}$ denote the endpoints of $\eta_i$, where $P_i$ (respectively, $Q_i$)
corresponds to $P$ (respectively, to $Q$). The numbering is chosen so that  $P_1$ and $Q_2$ are the zeros of $\omega'$.
The involution $\tau$ of $X$ descends to an involution of $X'$ exchanging $\eta_1$ and $\eta_2$. We denote this involution by
$\tau'$. Remark that $\tau'$ has two fixed points in $X'$, none of which are contained in the segments $\eta_1,\eta_2$. Note also that as  $(X,\omega)$ moves in its leaf of the kernel foliation, $(X',\omega')$ also moves in its leaf of the kernel foliation in $\H(1,1)$ (only the relative periods change).

Let $\iota$ be the hyperelliptic involution of $X'$. Since $\iota$ has six fixed points but $\tau'$ has two, we have $\iota \neq \tau$. Remark that
$\iota\circ\tau$ is also an involution of $X'$ satisfying $(\iota\circ\tau)^*\omega'=\omega'$. The surface $X''=X'/\langle \iota\circ\tau \rangle$
is an elliptic curve. Let $\pi$ is the branched covering $\pi: X'\rightarrow X''$, which is is ramified at $P_1$ and $Q_2$. Then
$\omega'$ descends to a holomorphic $1$-form $\omega''$ on $X''$ so that $\omega'=\pi^*\omega''$. For $i=1,2$ let
$P''_i,Q''_i,\eta''_i$ denote the images of $P_i,Q_i,\eta_i$ in $X''$. Note that we have $\omega''(\eta''_1)=\omega''(\eta''_2)=\omega(\s_0)$.

We consider the tuple $(X'',\omega'',P''_1,P''_2)$ as an element in $\H(0,0)$, that is the moduli space of flat tori with two marked points.
We first observe that as $(X,\omega)$ moves in its leaf of the kernel foliation, the corresponding surfaces $(X'',\omega'',P''_1,P''_2)$
are the same in $\H(0,0)$ (only $\omega''(\eta''_1)=\omega''(\eta''_2)$ change).
Indeed all the coordinates of $(X'',\omega'',P''_1,P''_2)$ are determined by the absolute periods of
$(X,\omega)$.

Let $\alpha''_1$ be a simple closed geodesic of $(X'',\omega'')$ which passes through $P''_1$ and does not contain $P''_2$.
Using $\GL(2,\R)$, we can assume that $\alpha''_1$ is horizontal.  By moving in the kernel foliation leaf of $(X,\omega)$, we can also
assume that $\eta''_i$ are parallel to $\alpha''_1$ ($\omega''(\eta''_i)= \lambda\omega''(\alpha''_1)$, with $0<\lambda <1$). By construction, the surface $(X',\omega')$ admits
a decomposition into cylinders in the horizontal direction. Note that $X'$ must have three horizontal cylinders, otherwise there would be
horizontal saddle connection joining $P_1$ to $Q_2$, which is excluded since $\alpha''_1$ does not contain $\eta''_2$ (see Figure~\ref{fig:twinsSC:AC}).
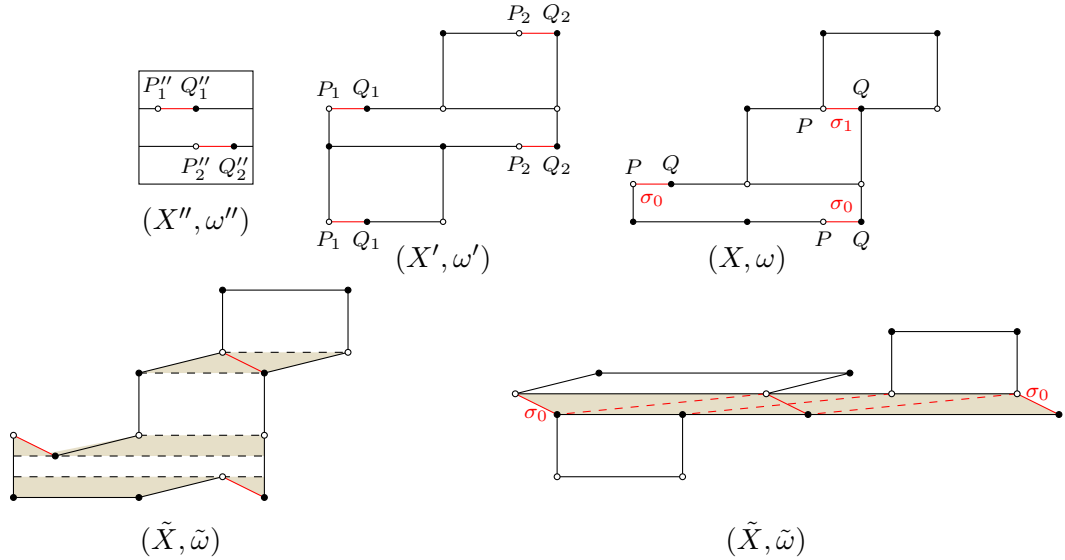
\begin{figure}[htb]
\subfloat{
\centering
\begin{tikzpicture}[scale=0.5]
\draw (-11,-1) -- +(3,0) -- +(3,3) -- +(0,3) -- cycle;
\draw (-11,1) -- +(0.5,0) (-9.5,1) -- +(1.5,0); \draw (-11,0) -- +(1.5,0) (-8.5,0)-- +(0.5,0);
\draw[red] (-10.5,1) -- (-9.5,1) (-9.5,0) -- (-8.5,0);

\draw (-5,-2) -- (-3,-2) -- (-3,0) -- (-1,0)  (0,0) -- (0,3) (-1,3) -- (-3,3) -- (-3,1) -- (-5,1) (-6,1) -- (-6,-2);
\draw (-6,0) -- (-3,0) (-3,1) -- (0,1);
\draw[red] (-6,1) -- (-5,1) (-6,-2) -- (-5,-2) (-1,3) -- (0,3) (-1,0) -- (0,0);

\draw (2,-1) -- (2,-2) -- (7,-2) (8,-2) -- (8,1) -- (10,1) -- (10,3) -- (7,3) -- (7,1) -- (5,1) -- (5,-1) -- (3,-1) (5,-1) -- (8,-1);
\draw[red] (2,-1) -- (3,-1) (7,-2) -- (8,-2) (7,1) -- (8,1);

\foreach \x in {(-10.5,1), (-9.5,0), (-6,1), (-3,1), (0,1), (-6,-2), (-3,-2), (-1,0),(-1,3), (2,-1), (5,-1), (7,-2), (7,1), (8,-1), (10,1)} \filldraw[fill=white] \x circle (2pt);

\foreach \x in {(-9.5,1), (-8.5,0), (-6,0), (-3,0), (0,0), (-3,3), (0,3),(-5,-2), (-5,1), (2,-2), (3,-1), (5,-2), (5,1), (7,3), (8,-2), (8,1), (10,3)} \filldraw[fill=black] \x circle (2pt);

\draw (-10.5,1) node[above] {$\scriptstyle P''_1$} (-9.5,1) node[above] {$\scriptstyle Q''_1$} (-9.5,0) node[below] {$\scriptstyle P''_2$} (-8.5,0) node[below] {$\scriptstyle Q''_2$};

\draw (-6,1) node[above] {$\scriptstyle P_1$} (-5,1) node[above] {$\scriptstyle Q_1$} (-6,-2) node[below] {$\scriptstyle P_1$} (-5,-2) node[below] {$\scriptstyle Q_1$} (-1,0) node[below] {$\scriptstyle P_2$} (-1,3) node[above] {$\scriptstyle P_2$} (0,0) node[below] {$\scriptstyle Q_2$}  (0,3) node[above] {$\scriptstyle Q_2$};

\draw (2,-1) node[above] {$\scriptstyle P$} (3,-1) node[above] {$\scriptstyle Q$} (7,-2) node[below] {$\scriptstyle P$} (7,1) node[below left] {$\scriptstyle P$} (8,-2) node[below] {$\scriptstyle Q$} (8,1) node[above] {$\scriptstyle Q$};

\draw[red] (2.5,-1) node[below] {$\scriptstyle \s_0$} (7.5,-2) node[above] {$\scriptstyle \s_0$} (7.5,1) node[below] {$\scriptstyle \s_1$};
\draw (-9.5,-2) node {$(X'',\omega'')$} (-3,-3) node {$(X',\omega')$} (5,-3) node {$(X,\omega)$};
\end{tikzpicture}
}
\qquad
\subfloat{
\centering
\begin{tikzpicture}[scale=0.55]
\fill[yellow!70!blue!30] (-9,0.5) -- (-6,0.5) -- (-4,1) -- (-7,1) -- cycle;
\fill[yellow!70!blue!30] (-12,-1) -- (-12,-1.5) -- (-11,-1.5) -- cycle;
\fill[yellow!70!blue!30] (-11.5,-1.5) -- (-6,-1.5) -- (-6,-1) -- (-9,-1) -- cycle;
\fill[yellow!70!blue!30] (-12,-2) -- (-12,-2.5) -- (-9,-2.5) -- (-7,-2) -- cycle;
\fill[yellow!70!blue!30] (-7,-2) -- (-6,-2.5) -- (-6,-2) -- cycle;
\fill[yellow!70!blue!30] (0,0) -- (1,-0.5) -- (13,-0.5) -- (12,0) -- cycle;

\draw (-12,-1) -- (-12,-2.5) -- (-9,-2.5) -- (-7,-2) (-6,-2.5) -- (-6,0.5) -- (-4,1) -- (-4,2.5) -- (-7,2.5) -- (-7,1) -- (-9,0.5) -- (-9,-1) -- (-11,-1.5);

\foreach \x in {(-12,-1), (-7,-2), (-7,1)} \draw[red] \x -- +(1,-0.5);
\foreach \x in {(-9,-1), (-9,0.5), (-7,1)} \draw[dashed] \x -- +(3,0);
\draw[dashed] (-12,-1.5) -- (-6,-1.5) (-12,-2) -- (-6,-2);

\foreach \x in {(-12,-1), (-9,-1), (-7,-2), (-7,1), (-6,-1), (-4,1)} \filldraw[fill=white] \x circle (2pt);
\foreach \x in {(-12,-2.5), (-11,-1.5), (-9,-2.5), (-9,0.5), (-7,2.5), (-6,-2.5), (-6,0.5), (-4,2.5)} \filldraw[fill=black] \x circle (2pt);

\draw (1,-0.5) -- (1,-2) -- (4,-2) -- (4,-0.5) -- (7,-0.5) -- (13,-0.5) (12,0) -- (12,1.5) -- (9,1.5) -- (9,0) -- (6,0) -- (8,0.5) -- (2,0.5) -- (0,0);
\draw (0,0) -- (6,0) (9,0) -- (12,0) (1,-0.5) -- (4,-0.5);
\foreach \x in {(0,0), (6,0), (12,0)} \draw[red] \x -- +(1,-0.5);
\foreach \x in {(6,0), (9,0), (12,0)} \draw[dashed, red] \x -- +(-5,-0.5);
\foreach \x in {(0,0), (1,-2), (4,-2), (6,0), (9,0), (12,0)} \filldraw[fill=white] \x circle (2pt);
\foreach \x in {(1,-0.5), (2,0.5), (4,-0.5), (7,-0.5), (8,0.5), (9,1.5), (12,1.5), (13,-0.5)} \filldraw[fill=black] \x circle (2pt);

\draw[red] (0.5,-0.5) node {$\scriptstyle \s_0$} (12.5,0) node {$\scriptstyle \s_0$};

\draw (6,-3.5) node {$(\tilde{X},\tilde{\omega})$} (-8,-3.5) node {$(\tilde{X},\tilde{\omega})$};
\end{tikzpicture}
}

\caption{$\s_0$ and $\s_1$ are invariant by $\tau$, $X'$ admits two involutions: the hyperelliptic one $\iota$ which fixes each of the cylinders, and the involution $\tau'$ induced by $\tau$ which exchanges the pair simple cylinders and fixes the larger one. Observe that $\tau'$ exchanges $\eta_1=\overline{P_1Q_1}$ and $\eta_2=\overline{P_2Q_2}$. }
\label{fig:twinsSC:AC}
\end{figure}

We can reconstruct $(X,\omega)$ from $(X',\omega')$: one sees that $(X,\omega)$ also admits a decomposition into three horizontal cylinders
(see Figure~\ref{fig:twinsSC:AC}). Consider the surface $(\tilde{X},\tilde{\omega})=(X,\omega)+(0,-\epsilon)$, with $\epsilon >0$ small as shown in Figure~\ref{fig:twinsSC:AC}.
We see that $(\tilde{X},\tilde{\omega})$ admits a decomposition into four horizontal cylinders, three of which are simple. It is easy to check
that there exists a triple of twin saddle connections $\gamma_0,\gamma_1,\gamma_2$ in the largest horizontal cylinder of
$\tilde{X}$ (which is preserved by the Prym involution) which satisfy $\tau(\gamma_0)=-\gamma_0$, $\tau(\gamma_1)=-\gamma_2$,
and $\gamma_1\cup\gamma_2$ is a separating  curve. This proves the lemma.
\end{proof}

\begin{Lemma}
\label{lm:twinsSC:CA:112}
Let $(X,\omega)\in \Prym(1,1,2)$ be a Prym eigenform having a double twin $\s_1$ of $\s_0$.  Then
one can find in the connected component of $(X,\omega)$ a surface having two pairs of homologous saddle connections $(\s'_1,\s''_1)$  and $(\s'_2,\s''_2)$ that are exchanged by the Prym involution.
\end{Lemma}
\begin{proof}
Let $(X',\omega') \in \H(1,1)$ be the surface obtained by the ``cutting-gluing'' construction along $c=\s_0*\tau(\s_0)*(-\s_1)$ (see Lemma~\ref{lm:cut}). Note that we have on $X'$ two marked and disjoint geodesic segments $\eta_1$ and $\eta_2$ (corresponding to $c$) such that the midpoint of $\eta_i$ is a zero of $\omega'$, and $\omega'(\eta_i)=\omega(\s_1)=2\omega(\s_0)$. We will consider $\eta_i$ as a slit with two sides $\eta'_i$ and $\eta_i''$, where $\eta'_i$ corresponds to $\s_1$, and $\eta''_i$ corresponds to the union $\s_0\cup\tau(\s_0)$. Remark that the Prym involution $\tau$ induces an involution $\tau'$ on $X'$ which is not the hyperelliptic involution. It follows that $X'$ is a double cover of a torus.

\begin{figure}[htb]
\centering
\begin{tikzpicture}[scale=0.6]
\draw  (-2.5,2) --  (-6,2) -- (-6,-1.5) -- (-3.5,-1.5) -- (-3.5,0) (-4.5,0) -- (-1,0) -- (-1,3.5) -- (-3.5,3.5) -- (-3.5,2);

\draw[dashed] (-6,0) -- (-4.5,0) (-2.5,2) -- (-1,2);

\draw (-4.5,2) .. controls (-4,1.7) and (-3,1.7) .. (-2.5,2);
\draw (-4.5,0) .. controls (-4,0.3) and (-3,0.3) .. (-2.5,0);

\foreach \x in {(-4.5,2), (-4.5,0), (-4.5,-1.5)} \filldraw[fill=white] \x circle (2pt);
\foreach \x in {(-2.5,3.5), (-2.5,2), (-2.5,0)} \filldraw[fill=black] \x circle (2pt);

\foreach \x in {(-6,2), (-6,0), (-6,-1.5), (-3.5,3.5), (-3.5,2), (-3.5,0), (-3.5,-1.5), (-1,3.5), (-1,2), (-1,0)} {\filldraw[fill=white] \x circle (3pt); \draw \x +(45:3pt) -- +(225:3pt) +(135:3pt) -- +(315:3pt);}

\draw (-3,1.5) node {\tiny $\eta'_1$} (-4,0.5) node {\tiny $\eta'_2$} (-5.25,2.3) node {\tiny $\s'_1$} (-1.75, 2.3) node {\tiny $\s''_2$} (-5.25,0.3) node {\tiny $\s''_1$} (-1.75, 0.3) node {\tiny $\s'_2$};

\draw (1,0) -- (1,-1.5) -- (3.5,-1.5) -- (3.5,0) -- (7,0) -- (7,2) -- (8,2) -- (8,3.5) -- (5.5,3.5) -- (5.5,2) -- (2,2) -- (2,0) -- cycle;
\draw (2,0) -- (3.5,0)  (5.5,2) -- (7,2) ;

\foreach \x in {(1,0), (1,-1.5), (3.5,2), (3.5,0), (3.5,-1.5)} \filldraw[fill=white] \x circle (2pt);
\foreach \x in {(5.5,3.5), (5.5,2), (5.5,0), (8,3.5), (8,2)} \filldraw[fill=black] \x circle (2pt);
\foreach \x in {(2,2), (2,0), (2,-1.5), (7,3.5), (7,2), (7,0)} {\filldraw[fill=white] \x circle (3pt); \draw \x +(45:3pt) -- +(225:3pt) +(135:3pt) -- +(315:3pt);}

\draw (1.5,0.3) node {\tiny $\s_0$} (7.5,3.8) node {\tiny $\tau(\s_0)$} (4.5,2.2) node {\tiny $\s_1$} (4.5,-0.2) node {\tiny $\s_1$} (2.75, 2.3) node {\tiny $\s'_1$} (2.75, 0.3) node {\tiny $\s''_1$} (6.25,2.3) node {\tiny $\s''_2$} (6.25,0.3) node {\tiny $\s'_2$};

\draw (-3.5,-2) node {\tiny $(X',\omega')$} (4.5,-2) node {\tiny $(X,\omega)$};
\end{tikzpicture}

\caption{$\Prym(1,1,2)$ case C: $\s_0$ has a double-twin.}
\label{fig:2xtwinSC:112}
\end{figure}
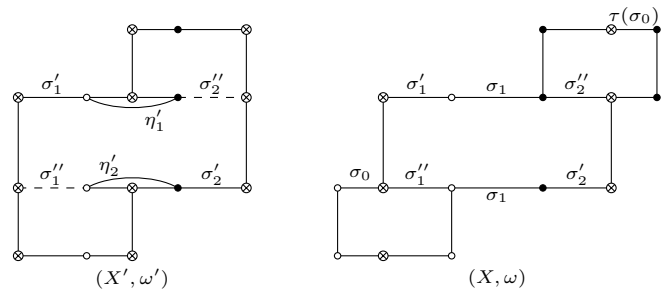
By construction, as $(X,\omega)$ moves in its kernel foliation leaf $(X',\omega')$ is fixed, only the marked geodesic segments (slits) $\eta_i$ are changed. Using $\GL^+(2,\R)$, we can assume that $X'$ is horizontally periodic, and it has  three horizontal cylinders (one may use the fact that $X'$ is a double cover of a torus to show that a periodic direction with three cylinder does exist).  We can arrange so that the slits are also horizontal, and $\eta'_i$ are contained in the boundary of the largest cylinder (see Figure~\ref{fig:2xtwinSC:112}). Reconstruct $(X,\omega)$ from $(X',\omega')$, we see that $(X,\omega)$ has two pairs of homologous saddle connections $(\s'_1,\s''_1)$ and $(\s'_2,\s''_2)$, such that $\tau(\s'_1)=-\s'_2, \tau(\s''_1)=\s''_2$,  and $\s'_1*\s_1*\s''_2$ is homologous to the core curve of the largest horizontal cylinder in $X$. The lemma is then proved.
\end{proof}

\section{Proof of the main result}
\label{sec:proof}

Let us now give the proof of our main theorem.
\begin{proof}[Proof of Theorem~\ref{thm:main}]
First of all $D \in \N, \; D \equiv 0,1,4 \mod 8$, the loci $\PrD(\kappa)$ are non empty: this is Corollary~\ref{cor:non:empty}.\medskip

We now consider the cases $D=0,1,4 \mod 8, D\geq 9, \text{ and } D \not\in\{9,16\}$.
By Theorem~\ref{thm:no:admisSC} and Corollary~\ref{cor:twin}, any component of $\Omega E_D(\kappa)$ contains a surface
with an admissible saddle connection that collapse to a point in $\Omega E_D(4)$. Recall that
$\Omega E_D(4)$ is a finite collection of Teichm\"uller discs. By Proposition~\ref{prop:neighbor}
for any connected component $\mathcal{C}$ of $\Omega E_D(4)$, there exists at most one component of
$\Omega E_D(\kappa)$ adjacent to $\mathcal{C}$ {\em i.e.} its closure contains $\mathcal{C}$.
Therefore, the number of connected components of $\PrD(\kappa)$ is bounded from above by the number of
components of $\PrD(4)$. \medskip

In particular, when $D\equiv 0,4 \mod 8$, since $\PrD(4)$ is connected (see Theorem~\ref{thm:comp:Prym4}), so is $\PrD(\kappa)$.
For $D\equiv 1 \mod 8$, by Theorem~\ref{thm:comp:Prym4} we know that $\PrD(4)$ has two components, so $\PrD(\kappa)$ has at most two
components. On the other hand, Theorem~\ref{thm:Dodd:disconnect} tells us that $\PrD(\kappa)$ cannot be connected.
Thus we can conclude that $\PrD(\kappa)$ has exactly two connected components. \medskip

For $ D\equiv 5 \mod 8$, if $\PrD(\kappa)$ is non-empty then again Theorem~\ref{thm:no:admisSC}, Corollary~\ref{cor:twin} and
Proposition~\ref{prop:collapse} implies that  $\PrD(4)$ is also non-empty, which contradicts Theorem~\ref{thm:comp:Prym4}. \medskip

We now consider the cases $D\in\{9,16\}$. Since
$\PrD(4)= \varnothing$, by Proposition~\ref{prop:collapse}, there exists no admissible saddle connection on any surface in $\Omega E_9(\kappa)\sqcup \Omega E_{16}(\kappa)$. By Theorem~\ref{thm:no:admisSC}, we see that any surface in $\Omega E_9(\kappa)\sqcup\Omega E_{16}(\kappa)$ belongs to the
same component as one of the surfaces
$$
\Sur_\kappa(1,1,-1), \Sur_\kappa(1,1,1), \Sur_\kappa(1,2,0) \textrm{ or } \Sur_{(2,2)}(2,1,0)
$$
(see Lemma~\ref{lm:example:D} for the definition).

It follows immediately that $\Omega E_{16}(2,2)$ and $\Omega E_9(\kappa)$ has at most two components
and $\Omega E_{16}(1,1,2)$ is connected.
The fact that $\Omega E_9(\kappa)$ is not connected is proved in Theorem~\ref{thm:Dodd:disconnect}. Hence
$\Omega E_{9}(\kappa)$ has exactly two connected components. \medskip

It remains to prove that $\Omega E_{16}(2,2)$ is connected. It is sufficient to show that
$\Sur_{(2,2)}(1,2,0), \Sur_{(2,2)}(2,1,0)\in \Omega E_{16}(2,2)^{\rm odd}$ belong to the same component.
We consider $(X_\eps,\omega_\eps)=\Sur_{(2,2)}(2,1,0)+(0,\varepsilon)$, with $\varepsilon >0$ small enough (see Figure~\ref{fig:D16:connect}).

\begin{figure}[htb]
\centering
\begin{tikzpicture}[scale=0.6]
\draw (2,-0.5) -- (3,-1) -- (3,1) -- (2,1.5) -- (2,2.5) -- (3,2) -- (3,3) -- (2,3.5) (1,3) -- (1,2) -- (0,2.5) -- (0,1.5) -- (1,1) -- (1,-1);
\foreach \x in {(1,3), (1,2), (1,1), (1,-1)} \draw[red] \x -- +(1,0.5);
\foreach \x in { (1,3), (3,3), (1,2), (3,2), (1,1), (3,1), (1,-1), (3,-1)} \filldraw[fill=white] \x circle (2pt);
\foreach \x in {(2,3.5), (0,2.5), (2,2.5), (0,1.5), (2,1.5), (2,-0.5)} \filldraw[fill=black] \x circle (2pt);
\draw (1.5,-1.5) node[below] {$\scriptstyle (X_\eps,\omega_\eps)=\Sur_{(2,2)}(2,1,0)+(0,\epsilon)$};

\draw (5,1) -- (6,-0.5) -- (8,-0.5) -- (7,1) -- (9,1) -- (10,0.5) -- (12,0.5) -- (11,1) -- (12,1.5) -- (10,1.5) -- (9,2) -- (7,2) -- (8,1.5) -- (6,1.5);
\foreach \x in {(5,1), (7,1), (9,1), (11,1)} \draw[red] \x -- +(1,0.5);
\foreach \x in {(5,1), (8,1.5), (9,1)} \draw \x -- +(2,0);
\foreach \x in {(7,2), (9,2), (5,1), (7,1), (9,1), (11,1)} \filldraw[fill=white] \x circle (2pt);
\foreach \x in {(6,1.5), (8,1.5), (10,1.5), (12,1.5), (6,-0.5), (8,-0.5), (10,0.5), (12,0.5)} \filldraw[fill=black] \x circle (2pt);
\draw (8.5,-1.5) node[below] {$\scriptstyle (X_\eps,\omega_\eps)$};

\draw (14,1.5) -- (14,-0.5) -- (16,-0.5) -- (16,1) -- (18,1) -- (18,0.5) -- (20,0.5) -- (20,1.5) -- (18,1.5) -- (18,2) -- (16,2) -- (16,1.5) -- cycle;
\foreach \x in {(14,1), (16,1.5), (18,1)} \draw \x -- +(2,0);
\foreach \x in {(14,1), (16,1), (18,1), (20,1)} \draw[red] \x -- +(0,0.5);
\foreach \x in {(16,2), (18,2), (14,1), (16,1), (18,1), (20,1)} \filldraw[fill=white] \x circle (2pt);
\foreach \x in {(14,1.5), (16,1.5), (18,1.5), (20,1.5), (14,-0.5), (16,-0.5), (18,0.5), (20,0.5)} \filldraw[fill=black] \x circle (2pt);
\draw (17,-1.5) node[below] {$\scriptstyle (Y,\eta)$};

\end{tikzpicture}
\caption{Connecting $\Sur_{(2,2)}(2,1,0)$ to $\Sur_{(2,2)}(1,2,0)$.}
\label{fig:D16:connect}
\end{figure}
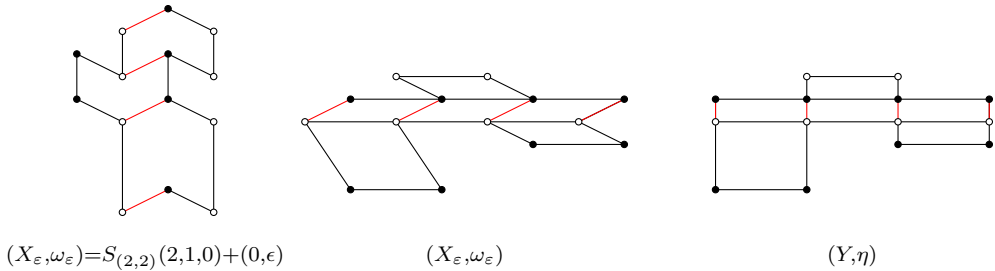
Observe that $(X_\eps,\omega_\eps)$ admits a decomposition into four horizontal cylinders. Moving horizontally in the kernel foliation leaf of
$(X_\eps,\omega_\eps)$ we get a surface $(Y,\eta)=(X_\eps,\omega_\eps)+v$, with $v\in \R\times\{0\}$, which admits a decomposition into three vertical cylinders. It is not difficult to see that $(Y,\eta)$ can be connected to $\Sur_{(2,2)}(1,2,0)$ by using the action of $\GL^+(2,\R)$ and moving in the kernel foliation leaves.
The proof of Theorem \ref{thm:main} is now complete.
\end{proof}

As a direct corollary we prove Theorem~\ref{thm:exist:3tori:dec} {\em i.e.} the existence in any component of
$\Omega E_D(\kappa)$ of surfaces which admit three-tori decompositions.

\begin{proof}[Proof of Theorem~\ref{thm:exist:3tori:dec}]
Let $(w,h,e)\in \Z^3$ be as in Lemma~\ref{lm:example:D} where $D=e^2+8wh$. We consider
the corresponding surfaces $(X_{\pm},\omega_\pm):=\Sur_\kappa(w,h,\pm e)$.
By Lemma~\ref{lm:example:D} $(X_{\pm},\omega_\pm)\in\Omega E_D(\kappa)$. \medskip

\noindent If $D \not \equiv 1 \mod 8$ then by Theorem~\ref{thm:main}, $\Omega E_D(\kappa)$ is connected
and $(X_\pm,\omega_\pm)$ admits a three-tori decomposition. \medskip

\noindent If $D \equiv 1 \mod 8$ then by Theorem~\ref{thm:main}, $\Omega E_D(\kappa)$ has
two connected components and from the proof of Theorem~\ref{thm:Dodd:disconnect},
$(X_{+},\omega_+)$ and $(X_{-},\omega_-)$ do not belong to the same connect component.
This ends the proof of Theorem~\ref{thm:exist:3tori:dec}.
\end{proof}


\end{document}